\documentclass{article}

\usepackage{amsmath,amssymb,amsthm,mathrsfs,color,times,textcomp,sectsty,verbatim}
\usepackage{xcolor}
\usepackage[colorlinks=true]{hyperref}
\hypersetup{urlcolor=blue, citecolor=red, linkcolor=blue}
\usepackage[colorinlistoftodos]{todonotes}
\usepackage[toc,page]{appendix}
\numberwithin{equation}{section}
\theoremstyle{plain}
\newtheorem{theorem}{Theorem}[section]
\newtheorem{proposition}[theorem]{Proposition}
\newtheorem{lemma}[theorem]{Lemma}
\newtheorem*{conjecture}{Conjecture}
\newtheorem{corollary}[theorem]{Corollary}
\newtheorem*{claim}{Claim}

\theoremstyle{definition}
\newtheorem{definition}[theorem]{Definition}

\newtheorem*{convention}{\textbf{Convention}}
\newtheorem{remark}[theorem]{Remark}

\newcommand\backmatter{\appendix
\def\chaptermark##1{\markboth{%
\ifnum  \c@secnumdepth > \m@ne  \@chapapp\ \thechapter:  \fi  ##1}{%
\ifnum  \c@secnumdepth > \m@ne  \@chapapp\ \thechapter:  \fi  ##1}}%
\def\sectionmark##1{\relax}}
\def \no{\nonumber}
\def \pa{\partial}

\def\e{\epsilon}

\def\R{\mathbb{R}}
\def\Rn{{\mathbb{R}}^n_+}
\def\d{\partial}

\def\a{\alpha}
\def\b{\beta}

\def\crit {\frac{2n}{n-2}}

\def\ba{\begin{align}}
\def\ea{\end{align}}
\def\bp{\begin{proof}}
\def\ep{\end{proof}}

\def\ubar{\bar{U}_{(x_0,\e)}}

\def\func:u{\bar{u}_{(x_0,\e)}}

\def\U{W_{\epsilon}}

\def\lp{\left(}
\def\rp{\right)}

\def\w{\psi}
 \allowdisplaybreaks

\begin{document}

\title{\Large \bf
The Han-Li conjecture in constant scalar curvature and constant boundary mean curvature problem on compact manifolds}
\author{Xuezhang  Chen\thanks{X. Chen:xuezhangchen@nju.edu.cn; $^\dag$Y. Ruan:ruanyp@umich.edu; $^\ddag$L. Sun:lsun@math.jhu.edu.}~, Yuping Ruan$^\dag$ and Liming Sun$^\ddag$\\
 \small
$^\ast$$^\dag$Department of Mathematics \& IMS, Nanjing University, Nanjing
210093, P. R. China\\
\small
$^\ddag$Department of Mathematics, Johns Hopkins University, Baltimore, MD, 21218, USA
}

\date{}

\maketitle

\begin{abstract}
The Han-Li conjecture states that: Let $(M,g_0)$ be an $n$-dimensional $(n\geq 3)$ smooth compact Riemannian manifold with boundary having positive (generalized) Yamabe constant and $c$ be any real number, then there exists a conformal metric of $g_0$ with scalar curvature $1$ and boundary mean curvature $c$. Combining with Z. C. Han and Y. Y. Li's results, we answer this conjecture affirmatively except for the case that $n\geq 8$, the boundary is umbilic, the Weyl tensor of $M$ vanishes on the boundary and has a non-zero interior point.

{{\bf $\mathbf{2010}$ MSC:} Primary 53C21, 58J32, 35J20; Secondary 58J05, 35B33, 34B18.}

{{\bf Keywords:} Mountain Pass Lemma, conformal Fermi coordinates, boundary Yamabe problem, positive mass theorem.}
\end{abstract}

\tableofcontents

\section{Introduction}\label{Sect1}
On a smooth compact Riemannian manifold with boundary, the analogues of the Yamabe problem were initially studied by J. Escobar. Over the past more than twenty years, such problems have been extensively investigated by numerous researchers. In some literatures, these problems are also referred to as \textit{Escobar problem}. We give a brief summary for these problems. It is convenient to distinguish into three cases.  The first case is concerned with the existence of conformal metrics with constant scalar curvature and zero boundary mean curvature. This problem was initially studied by J. Escobar \cite{escobar4} in the case of $3\leq n\leq 5$ or $n\geq 6$ and the boundary has a non-umbilic point, later by S. Brendle-S. Chen \cite{Brendle-Chen} in the case of $n \geq 6$ and the boundary is umbilic, assuming the validity of the Positive Mass Theorem (PMT). For recent associated curvature flows, readers are referred to \cite{Brendle1, ChenHo, Almaraz-Sun} and the references therein. The second case is concerned with the existence of scalar-flat conformal metrics with constant boundary mean curvature under the condition that the corresponding generalized Yamabe constant is finite. It was also first studied by Escobar \cite{escobar1},  subsequently in Escobar \cite{escobar6}, and then by F. Marques \cite{marques3} in the case of $3 \leq n \leq 5$ or $M$ is locally conformally flat with umbilic boundary or $n\geq 6$ and the boundary has at least one non-umbilic point.  Some remaining cases have been studied by Marques \cite{ marques1}, S. Chen \cite{ChenSophie} with assuming PMT, S. Almaraz \cite{Almaraz1} etc. More recently,  without the PMT, M. Mayer and C. Ndiaye in \cite{mayer-ndiaye} studied the remaining cases, but in general the solution they obtained is not a minimizer of the associated energy functional. See \cite{Brendle1, ChenHo,almaraz5} etc for the related conformal curvature flows. The third case is concerned with the existence of conformal metrics with (non-zero) constant scalar curvature and (non-zero) constant boundary mean curvature. For variational methods, see \cite{escobar3, escobar5, han-li1, han-li2, Araujo1, Araujo2, ChenHoSun, Chen-Sun} and the references therein. For flow approaches, see \cite{ChenHoSun,ChenHoSun1}. From now on, we focus on the last case. This paper can be regarded as a continuation of the first and third named authors' paper \cite{Chen-Sun}.

Let $(M,g_0)$ be a smooth compact Riemannian manifold of dimension $n\geq 3$  with boundary $\partial M$. The (generalized) Yamabe constant $Y(M,\partial M,[g_0])$ is defined as
\begin{equation*}
Y(M,\partial M,[g_0]):=\inf_{g\in[g_0]}\frac{\int_MR_gd\mu_g+2(n-1)\int_{\partial M}h_gd\sigma_g}{(\int_{M}d\mu_g)^{\frac{n-2}{n}}},
\end{equation*}
where $R_g$ is the scalar curvature of $M$ and $h_g$ is the mean curvature on $\pa M$ of metric $g$. Define
$$E[u]=\int_{M}\left(\frac{4(n-1)}{n-2}|\nabla u|_{g_0}^2+R_{g_0}u^2\right)d\mu_{g_0}+2(n-1)\int_{\partial M} h_{g_0}u^2d\sigma_{g_0}.$$
For $c \in \mathbb{R}$, we consider a ``free" functional
\begin{align}\label{main_funtional_I}
I[u]=E[u]-4(n-1)(n-2)\int_{M}u_+^{\frac{2n}{n-2}}d\mu_{g_0}-4c\int_{\partial M}u_+^{\frac{2(n-1)}{n-2}}d\sigma_{g_0}
\end{align}
for all $u \in H^1(M,g_0)$, where $u_+=\max\{u,0\}$. Then we can verify that $I\in C^2(H^1(M,g_0);\mathbb{R})$. It is not difficult to check that any non-trivial critical point $u$ of $I$ solves
\begin{align}\label{eq:main_eq_on_M}
\begin{cases}
\displaystyle-\frac{4(n-1)}{n-2}\Delta_{g_0}u+R_{g_0}u=4n(n-1)u_+^{\frac{n+2}{n-2}}&\mathrm{ in ~~} M,\\
\displaystyle\frac{\partial u}{\partial{\nu_{g_0}}}+\frac{n-2}{2}h_{g_0}u=c u_+^{\frac{n}{n-2}}&\mathrm{ on ~~}\partial M,
\end{cases}
\end{align}
where $\nu_{g_0}$ is the outward unit normal on $\pa M$ with respect to $g_0$. A simple application of the maximum principle and the Hopf boundary point lemma yields that such $u$ must be positive in $\overline M$. The regularity theory in \cite{Cherrier} shows that $u$ is smooth in $\overline{M}$. Thus if we let $g=u^{4/(n-2)}g_0$, \eqref{eq:main_eq_on_M} implies that $R_g=4n(n-1)$ and $h_g=2c/(n-2)$. For brevity, denote by $L_{g_0}=-\frac{4(n-1)}{n-2}\Delta_{g_0}+R_{g_0}$ the conformal Laplacian and $B_{g_0}=\frac{\partial}{\partial \nu_{g_0}}+\frac{n-2}{2}h_{g_0}$ the boundary conformally covariant operator, respectively. Both $L_{g_0}$ and $B_{g_0}$ have the following conformally invariant properties: Let $g=u^{4/(n-2)}g_0$, then for any $\varphi \in C^\infty(\overline M)$, there hold
\begin{equation}\label{eq:conformal_invariance}
L_{g_0}(u\varphi)=u^{\frac{n+2}{n-2}}L_g(\varphi) \mathrm{~~and~~} B_{g_0}(u\varphi)=u^{\frac{n}{n-2}}B_g(\varphi).
\end{equation}

In $1999$, Zheng-Chao Han and Yan Yan Li  \cite{han-li2} proposed the following conjecture.
\begin{conjecture}[\textbf {Han-Li}] \textit{If~ $Y(M,\pa M,[g_0])>0$, then given any $c\in \mathbb{R}$, problem \eqref{eq:main_eq_on_M} is solvable.}
\end{conjecture}

Compared with the previous two cases in \textit{Escobar Problem}, besides the critical Sobolev exponents coming from both equations of \eqref{eq:main_eq_on_M} in the interior and on the boundary, the arbitrariness of the constant $c$ leads to an additional difficulty in solving the Han-Li conjecture.

Furthermore, they confirmed their conjecture if any of the following hypotheses is fulfilled:
\begin{enumerate}
\item[(a)]$n \geq 5$ and $\pa M$ admits at least one non-umbilic point (see \cite{han-li1}); 
\item[(b)]$n\geq 3$ and $(M,g_0)$ is locally conformally flat with umbilic boundary $\pa M$ (see \cite{han-li2});
\item[(c)] $n=3$ (see \cite{han-li3}).
\end{enumerate}

Another natural motivation of the study of the Han-Li conjecture came from Escobar's work \cite{escobar5} in $1990$: Under what conditions can an Einstein manifold be conformally deformed to another Einstein manifold? A partial result of this problem was given in \cite[Theorem 2.1]{escobar5}, which demonstrates that an affirmative answer to the Han-Li conjecture is exactly a sufficient condition of this problem.

In \cite{Chen-Sun}, we studied the third case of \textit{Escobar problem} and partially answered Han-Li conjecture. We used subcritical approximations to find minimizers of $E[u]$ with a suitable homogeneous constraint (equivalently, a quotient functional) and established that under some natural hypotheses on manifolds there exists a conformal metric with constant scalar curvature $1$ and some positive constant boundary mean curvature. In particular, this constant allows to be very large. This provides us more evidences to the Han-Li conjecture. However, it is still difficult to solve this conjecture by seeking minimizers of the quotient functional, due to the reason that the Lagrange multiplier of Euler-Lagrange equation involves both the scalar curvature and the boundary mean curvature, and therefore it seems hard to get the arbitrariness of constant boundary mean curvature in the Han-Li conjecture. One way to get around the difficulty is to use the free functional \eqref{main_funtional_I}, which was introduced by Z. C. Han and Y. Y. Li \cite{han-li1}. 

Their strategy was to find a non-trivial  mountain pass critical point of $I$ via the following Mountain Pass Lemma of Ambrosetti and Rabinowitz \cite{Ambrosetti-Rabinowitz}.

\textbf{Mountain Pass Lemma (MPL)}. Let $X$ be a Banach space and $I\in C^1(X;\mathbb{R})$. Suppose that $I[0]=0$ and there exists $0\neq u_0\in X$ such that $I[u_0]\leq 0$. Let $\Gamma$ denote the set of continuous paths in $X$ connecting 0 and $u_0$ and define $I_{mp}:=\inf_{\gamma\in \Gamma}\sup_{u\in \gamma}I[u]$. Suppose that $I_{mp}>0$ and that $I$ satisfies the $(PS)$ condition at the level $I_{mp}$. Then $I_{mp}$ is a critical value of $I$.  

In general the $(PS)$ condition for the associated functional $I$ in \eqref{main_funtional_I} can not be satisfied  due to the critical nonlinearities of $I$. However, Z. C. Han and Y. Y. Li proved that $I$ satisfies a weak $(PS)$ condition at any energy level below a certain threshold $S_c$, but it is enough to prove the existence result. Here $S_c$ was given in \cite{han-li2}, as well as will be defined in \eqref{eq:S_c_n=4} below. 

The mountain pass structure of $I$ can be verified through the following facts. Given any $u\in H^1(M,g_0)$ with $u_+\not\equiv 0$, we define $f(t):=I[tu]$ for $t\in [0,\infty)$. It is not hard to show that $f(0)=0$ and $\lim_{t\to\infty}f(t)=-\infty$ and $f(t)$ admits a unique maximum point $t_\ast$ (see also Section \ref{Sect2}). Since $Y(M,\partial M,[g_0])>0$, it is not hard to show that there exists $\e_0>0$ such that $I[u]\geq \e_0$ when $\|u\|_{H^1(M,g_0)}=r_0$ for small enough $r_0>0$. For each $u \in H^1(M,g_0)$ with $u_+ \not \equiv 0$, there holds $I[tu]<0$ for sufficiently large $t>0$. We choose $u_0=t_0 u$ for some suitable $t_0$ such that $I[t_0 u_0]\leq 0$ and define $I_{mp}$ as in the statement of MPL. Indeed, with an improvement of the proof of the Han-Li's \cite[Lemma 1.2]{han-li2}, we can verify that the $(PS)$ condition for $I$ below the mentioned energy level $S_c$ is satisfied (see Lemma \ref{lem:compactness}). Then, if there exists some $u \in H^1(M,g_0)$ with $u_+ \not \equiv 0$ such that 
\begin{align}\label{ieq:goal_less}
\max_{t \in [0,\infty)} I[tu]<S_c,
\end{align}
then $0<\e_0\leq I_{mp}<S_c$. Hence, the existence of a nontrivial mountain pass critical point of $I$ is guaranteed by the $(PS)$ condition proved in Lemma \ref{lem:compactness} in Section \ref{Sect2}. Consequently, in order to verify the Han-Li conjecture, the goal of this paper is to find some appropriate test functions satisfying \eqref{ieq:goal_less}. Since the functional $I$ is conformally invariant, we sometimes work in conformal Fermi coordinates around a boundary point (see \cite{marques1}) to simplify the analysis. Due to various geometric nature of the compact manifold, the construction of a test function has both a local and a global aspects. The developments listed in the beginning of this introduction deepen our understanding on how the test functions should be constructed.

To state our main theorem, we let $d=[(n-2)/2]$ for $n \geq 3$ and define as in \cite{almaraz5,Chen-Sun}  
\begin{align*}
\mathcal{Z}=
\Big\{x_0\in\pa M;&\limsup_{x \in M, x\to x_0}d_{g_0}(x,x_0)^{2-d}|W_{g_0}(x)|_{g_0}=0\mathrm{~~and~~}\\
&\limsup_{x \in \pa M, x\to x_0}d_{g_0}(x,x_0)^{1-d}|\mathring{\pi}_{g_0}(x)|_{g_0}=0\Big\},
\end{align*}
where $W_{g_0}$ denotes the Weyl tensor of $M$, and $\pi_{g_0}, \mathring{\pi}_{g_0}$ denote the second fundamental form on $\pa M$ and its trace-free part, respectively.

Combining with the aforementioned Han-Li's existence results, we can solve Han-Li's conjecture in the affirmative under the following natural geometric assumptions.
\begin{theorem}\label{thm:main}
Let $(M,g_0)$ be a smooth compact Riemannian manifold of dimension $n\geq 3$  with boundary $\partial M$ and assume that $Y(M,\pa M,[g_0])>0$, then the Han-Li conjecture is true, provided that one of the following hypotheses is satisfied:
\begin{enumerate}
\item[(i)] $3\leq n\leq 7$;
\item[(ii)] $n \geq 8$, $\mathcal{Z}$ is non-empty and $M$ is spin;
\item[(iii)] $n \geq 8$, $\pa M$ is umbilic and the Weyl tensor of $M$ is non-zero at a boundary point;
\item[(iv)] $n \geq 8$, $M$ is locally conformally flat with umbilic boundary;
\item[(v)] $n \geq 8$, $\pa M$ has at least one non-umbilic point.
\end{enumerate}
\end{theorem}

\begin{remark}
For $c=0$ in the Han-Li conjecture, Theorem \ref{thm:main} can directly follow from \cite{escobar4,Brendle-Chen}, as well as \cite{Almaraz-Sun} from the viewpoint of a geometric flow. As mentioned earlier, the assertions \textit{(iv)} and \textit{(v)} are due to Z. C. Han and Y. Y. Li. 

\end{remark}
\begin{remark}
From Theorem \ref{thm:main}, we conclude that the Han-Li conjecture is true except for the case that $n \geq 8$, $\pa M$ is umbilic and the Weyl tensor of $M$ vanishes on $\pa M$ and has a non-zero interior point.
\end{remark}

We sketch the procedure of the proof of Theorem \ref{thm:main}.  When $\mathcal{Z}$ is non-empty, then by using the inverted coordinates near a point of $\mathcal{Z}$, one can obtain an asymptotically flat manifold with non-compact boundary. It follows from the PMT in \cite{almaraz-barbosa-lima} that the mass of this manifold unless being isometric to the half space is strictly positive when $3\leq n\leq 7$. Therefore the test function $\ubar$ constructed in the first and third author's previous paper \cite{Chen-Sun} (see also \eqref{test_fcn_Chen_Sun}) can be applied with some adaptation to all $c \in \mathbb{R}$. In particular $n=3$ will imply $\mathcal{Z}=\partial M$, we give an alternative proof for $n=3$. We mention that such ideas should be traced back to S. Brendle \cite{Brendle4,Brendle2}. Now we are left with the case $\mathcal{Z}=\emptyset$ in $4\leq n\leq 7$. Since $d=[(n-2)/2]\leq 2$, any boundary point $x_0\not\in \mathcal{Z}$ will either be non-umbilic or have non-vanishing Weyl tensor at the same point. We construct some local test functions, which are modifications of $\ubar$, for the remaining cases in dimensions $n=4,6$ and prove \eqref{ieq:goal_less} using them. Hence, we completely solve the Han-Li conjecture in dimension $3 \leq n\leq 6$, and dimension $n=7$ except for the case that $\pa M$ is umbilic and the Weyl tensor of $M$ is non-zero at a boundary point. However, the dimension $n=7$ and higher are more subtle, partly due to the loss of the term $\log(\rho/\e)$ in deriving \eqref{ieq:goal_less} as in dimensions $n=4,6$.  Assume that the same hypothesis in the remaining case in dimension $n=7$ also holds for $n \geq 8$, we succeed to prove the Han-Li conjecture.  Next we combine the following two different types of test functions to achieve our goal. One is still to use the test function $\ubar$, we dig into the error correction term to obtain some new estimates. It boils down to prove an inequality (see \eqref{est:goal_higher_dim}) for all dimensions $n \geq 7$, which has independent interest. Possibly due to technical reasons, we are now only able to prove it when $c$ is not less than a negative dimensional constant, which is given in Theorem \ref{thm:negative_c_higher_dim}. For the other range of $c$, we can construct another type of test function tailored to this case. In the work of F. Marques \cite{marques3} and S. Almaraz \cite{Almaraz1} on the second case of \textit{Escobar problem}, it is notable that their test functions have just one parameter and could not work here directly. Somewhat inspired by their ideas, with the use of three parameters $\kappa_0,\kappa_1,\kappa_2$, we explicitly construct a local test function defined in \eqref{test_fcn:non-positive_higher_dim}, still denoted by $\ubar$ for simplicity. We define $\kappa:=(\kappa_2,\kappa_1,\kappa_0,1)\in \mathbb{R}^4$. Through complex computations together with some delicate observations, we eventually arrive at 
$$\max_{t \in [0,\infty)} I[t\ubar]=I[t_\ast \ubar]=S_c+a_n \kappa \mathcal{Q} \kappa^\top |R_{nanb}|^2 \e^4 -b_n |\overline W|_{g_{x_0}}^2\e^4+\mathrm{h.o.t.}$$
for small enough $\e$, where $a_n,b_n$ are two positive constants,  $\mathcal{Q}$ is a real symmetric matrix, $g_{x_0}\in [g_0]$, $R_{nanb}$ is the Riemann curvature tensor, and $\overline{W}$ is the Weyl tensor of $\pa M$. See \eqref{est:error_final_higher_dim}. Based on Lemma \ref{lem:Almaraz}, our greatest challenging task is to find a vector  $\kappa$ above, such that $\kappa \mathcal{Q} \kappa^\top<0$. After many attempts, finally we succeed to find a ``\textit{good}" $\kappa=-\mathcal{V}\mathcal{S}_2^\top \mathcal{S}_1^\top(n-2)/2$ with $\mathcal{V}=(2/3,T_c,0,1)$, where $\mathcal{S}_1,\mathcal{S}_2$ are two non-singular matrices defined in \eqref{first_nonsing_matrix} and \eqref{second_nonsing_matrix}. More details are referred to the end of Subsection \ref{Subsect5.2}. Consequently, these two test functions match so perfectly that we can completely solve the Han-Li conjecture in the aforementioned case.

The main ingredients of our approach are as follows: Firstly, the test function constructed in our previous paper \cite{Chen-Sun} and some necessary estimates therein play an important role in our proof. Another observation is that such estimates  in \cite{Chen-Sun} of the test function for positive $c$ can be naturally extended to the ones for all real numbers $c$. Secondly, we used the precise expression \eqref{eq:qudratic_root_t} of the associated maximum $t_\ast$ with various test functions in the verification of condition \eqref{ieq:goal_less}. Thirdly, the conformal Fermi coordinate system \cite{marques1} greatly simplifies the computations and plays the same role as the conformal normal coordinate system on the closed manifolds. Finally, similar to the resolution of the Yamabe problem, the construction of a global test function heavily relies  on the PMT \cite{almaraz-barbosa-lima} with a non-compact boundary.

In our forthcoming paper \cite{ChenHoSun1}, we will employ a conformal curvature flow to study this constant scalar curvature and constant boundary curvature problem on a compact manifold.

The paper is organized as follows. In Section \ref{Sect2}, we collect some elementary properties of $I$ and set up some notations. In Section \ref{Sect3}, we establish the Han-Li conjecture when $\mathcal{Z}$ is non-empty and the PMT is valid, which together with Han-Li's existence results covers a large class of the lower dimensions $3 \leq n \leq 7$. See Theorem \ref{thm:Han-Li_lower_dims}. In Section $\ref{Sect4}$, the Han-Li conjecture is confirmed for the remaining cases of dimensions $n=4,6$. In Section \ref{Sect5}, we employ two different types of test functions mainly according to the sign of $c$. In Subsection \ref{Subsect5.1}, we prove that problem \eqref{eq:main_eq_on_M} is solvable for any constant not less than a negative dimensional constant, under the condition that the boundary is umbilic and the Weyl tensor of $M$ is non-zero at some boundary point. In Subsection \ref{Subsect5.2}, through constructing a local test function, we complete the proof of the Han-Li conjecture for the remaining case in dimension $n=7$, as well as for $n \geq 8$ with the same assumption. In Appendix \ref{Appen}, we give detailed proofs of Lemma \ref{lem:compactness} and a technical result used in Subsection \ref{Subsect5.1}.

\noindent{\bf Acknowledgements.} The authors would like to thank Professor Yan Yan Li for his fruitful discussions and encouragement. Part of this work was carried out during the first named author's several visits to Rutgers University. He also would like to thank the Mathematics Department at Rutgers University for its hospitality and financial support. He was supported by NSFC (No.$11771204$), A Foundation for the Author of National Excellent Doctoral Dissertation of China (No.$201417$), the start-up grant of $2016$ Deng Feng Program B at Nanjing University, the travel grants from AMS Fan fund and Hwa Ying foundation at Nanjing University.

\section{Preliminaries}\label{Sect2}

\indent \indent We collect some elementary properties of the standard bubble and $I[u]$, as well as a criterion of seeking the $(PS)$ condition for $I$, and introduce some notations at the end of this section.

Let $\mathbb{R}_+^n=\{y=(y^1,\cdots,y^n) \in \mathbb{R}^n; y^n>0\}$ be the Euclidean half-space. Define $T_c=-c/(n-2)$ and
\begin{equation}\label{eq:bdry_bubble}
\U(y)=\e^{\frac{2-n}{2}}W(\e^{-1}y)=\left(\frac{\e}{\e^2+|y-T_c\e \textbf{e}_n|^2}\right)^{\frac{n-2}{2}}
\end{equation}
for any $\e>0$, where $\textbf{e}_n$ is the unit direction vector in the $n$-th coordinate. Then $W_\e$ satisfies
\begin{align}\label{prob:half-space}
\begin{cases}
\displaystyle-\Delta W_\e=n(n-2) W_\e^{\frac{n+2}{n-2}},\quad &\mathrm{~~in~~}\R^n_+,\\
\displaystyle \frac{\pa W_\e}{\pa y^n}=(n-2)T_c W_\e^{\frac{n}{n-2}},\quad &\mathrm{~~on~~} \R^{n-1}.
\end{cases}
\end{align}
Readers are referred to \cite{li-zhu,escobar5} for the classification theorem of all positive solutions to \eqref{prob:half-space}.
Multiplying the equations in \eqref{prob:half-space} by $W_\e$, we integrate by parts to get
 \begin{align}\label{energy:bubble2}
\int_{\mathbb{R}^n_+}|\nabla W_\e(y)|^2dy=n(n-2)\int_{\mathbb{R}^n_+}W_\e(y)^{\frac{2n}{n-2}}dy+c\int_{\mathbb{R}^{n-1}}W_\e(y)^{\frac{2(n-1)}{n-2}}d\sigma.
 \end{align}
Throughout the paper, we set
 \begin{equation}\label{def:volumes}
 A=\int_{\mathbb{R}^n_+}W_\e(y)^{\frac{2n}{n-2}}dy\quad \mathrm{~~and~~}\quad B=\int_{\mathbb{R}^{n-1}}W_\e(y)^{\frac{2(n-1)}{n-2}}d\sigma.
\end{equation}
Then $A$ and $B$ are independent of $\e$.
Define
\begin{align}\label{def:S_c}
S_c:=&\frac{4(n-1)}{n-2}\int_{\mathbb{R}_+^n}|\nabla W_\e|^2dy-4(n-1)(n-2)\int_{\mathbb{R}_+^n}W_\e^{\frac{2n}{n-2}}dy\no\\
&-4c\int_{\mathbb{R}^{n-1}}W_\e^{\frac{2(n-1)}{n-2}}d\sigma\no\\
=&\frac{4}{n-2}\int_{\mathbb{R}^n_+}|\nabla W_\e(y)|^2dy+4(n-2)\int_{\mathbb{R}^n_+}W_\e^{\frac{2n}{n-2}}dy>0
\end{align}
by $\eqref{energy:bubble2}$. 
Notice that $S_c$ only depends on $n,c$ by virtue of  \eqref{energy:bubble2} and \eqref{def:S_c}, more concretely,
\begin{equation}\label{eq:S_c_n=4}
S_c=8(n-1)A+\frac{4}{n-2}cB.
\end{equation}

Given any $u \in H^1(M,g_0)$ with $u_+ \not \equiv 0$, for $t \in [0,\infty)$ we define
\begin{align*}
f(t):=&I[tu]\\
=&t^2E[u]-4(n-1)(n-2)\int_{M}u_+^{\frac{2n}{n-2}}d\mu_{g_0}t^{\frac{2n}{n-2}}-4c\int_{\partial M}u_+^{\frac{2(n-1)}{n-2}}d\sigma_{g_0}t^{\frac{2(n-1)}{n-2}}.
\end{align*}
We now claim that there exists a unique maximum point $t_\ast=t_\ast(u)$ of $f(t)$ in $[0,\infty)$, namely
\begin{align}\label{max_I_at_tu}
\max_{t \in [0,\infty)}I[tu]=\frac{1}{n-1}E[u]t_\ast^2+4(n-2)\int_M u_+^{\frac{2n}{n-2}}d\mu_{g_0} t_\ast^{\frac{2n}{n-2}},
\end{align}
where $t_\ast$ satisfies
\begin{align}\label{eq:constraint_max_t}
\frac{n-2}{4(n-1)}E[u]=n(n-2)\int_{M}u_+^{\frac{2n}{n-2}}d\mu_{g_0}t_\ast^{\frac{4}{n-2}}+c\int_{\partial M}u_+^{\frac{2(n-1)}{n-2}}d\sigma_{g_0}t_\ast^{\frac{2}{n-2}}.
\end{align}

To show this, we simplify $f(t)$ as
$$f(t)=at^2-\theta t^{\frac{2n}{n-2}}+bt^{\frac{2(n-1)}{n-2}},\quad t\in [0,\infty),$$
where $a,\theta>0$ and $b \in \mathbb{R}$. Then
$$f(0)=0 \mathrm{~~and~~} \lim_{t \to \infty} f(t)=-\infty.$$
We compute
\begin{align*}
f'(t)=&2at-\theta \frac{2n}{n-2}t^{\frac{n+2}{n-2}}+b\frac{2(n-1)}{n-2}t^{\frac{n}{n-2}},\\
f''(t)=&2a-\theta \frac{2n(n+2)}{(n-2)^2}t^{\frac{4}{n-2}}+b\frac{2(n-1)n}{(n-2)^2}t^{\frac{2}{n-2}}.
\end{align*}
Then there exists a unique $t_\ast \in (0,\infty)$ such that $f'(t_\ast)=0$. Moreover, we obtain
\begin{align*}
f''(t_\ast)=&-\theta \frac{8n}{(n-2)^2}t_\ast^{\frac{4}{n-2}}+b\frac{4(n-1)}{(n-2)^2}t_\ast^{\frac{2}{n-2}}\\
=&-\frac{2}{n-2}\left[\frac{2n}{n-2}\theta t_\ast^{\frac{4}{n-2}}+2a\right]<0.
\end{align*}
This implies that $t_\ast$ is the unique maximum point of $f(t)$ in $[0,\infty)$.

For simplicity, we use $\ubar$ to denote some qualified test function of \eqref{ieq:goal_less}, depending on small $\e>0$ and some suitable $x_0 \in \pa M$, which is non-negative, though it may vary in different sections. We also define
$$\mathcal{E}=\frac{n-2}{4(n-1)}E[\ubar],~ \mathcal{A}=n(n-2)\int_M\ubar^{\frac{2n}{n-2}}d\mu_{g_{x_0}},~ \mathcal{B}=c\int_{\partial M}\ubar^{\frac{2(n-1)}{n-2}}d\sigma_{g_{x_0}}.$$
From \eqref{max_I_at_tu}, \eqref{eq:constraint_max_t} and the assumption that $Y(M,\pa M,[g_0])>0$, we obtain
\begin{align}\label{eq:qudratic_root_t}
t^{\frac{2}{n-2}}_*=\frac{-\mathcal{B}+\sqrt{\mathcal{B}^2+4\mathcal{E}\mathcal{A}}}{2\mathcal{A}},
\end{align}
which is a positive constant depending on each $\ubar$.

The following lemma is a slightly improved version of Han-Li's \cite[Lemma 1.2]{han-li2} and provides the $(PS)$ condition for $I$ below the level $S_c$.
\begin{lemma}[\textbf{Compactness}]\label{lem:compactness} Suppose that $Y(M,\partial M,[g_0])>0$. Let $\{u_i;i \in \mathbb{N}\}$ be a sequence of functions in $H^1(M,g_0)$ satisfying $I[u_i]\to L<S_c$ and 
$$\max_{v\in H^1(M,g_0)\backslash\{0\}}\frac{|I'[u_i](v)|}{\|v\|_{H^1(M,g_0)}}\to 0 \quad\mathrm{~~as~~} i \to \infty.$$
Then after passing to a subsequence, either (i) $\{u_i\}$ strongly converges in $H^1(M,g_0)$ to some positive solution $u$ of \eqref{eq:main_eq_on_M} or (ii) $\{u_i\}$ strongly converges to $0$ in $H^1(M,g_0)$. 
\end{lemma}
Traced back to Han-Li's \cite[Lemma 1.2]{han-li2}, the original statement of Case \textit{(i)} is that $\{u_i\}$ weakly converges in $H^1(M,g_0)$ to some solution $u$ of \eqref{eq:main_eq_on_M}. Inspired by Han-Li's idea, we present a proof in Appendix \ref{Appen} to conclude Lemma \ref{lem:compactness}.

\begin{convention} 
Let $a,b,c,\cdots$ range from $1$ to $n-1$ and $i,j,k,\cdots$ from $1$ to $n$. For small $\rho>0$, let
\begin{align*}
B_\rho^+(0)=B_\rho(0)\cap \mathbb{R}^n_+;\qquad &\partial^+ B_\rho^+(0)=\partial B_\rho^+(0)\cap \mathbb{R}^n_+;\\
D_\rho(0)=\partial B_\rho^+(0)\backslash \partial^+ B_\rho^+(0)&
\end{align*}
and simplify $B_\rho^+(0),\partial^+ B_\rho^+(0)$, $D_\rho(0)$ by $B_\rho^+$, $\partial^+ B_\rho^+$, $D_\rho$ without otherwise stated. 

Given any $x_0 \in \pa M$ and any integer  $N \geq 1$, it follows from F. Marques \cite[Proposition 3.1]{marques1} that there exists $g_{x_0} \in [g_0]$ such that under $g_{x_0}$-Fermi coordinates $\{(y^1,\cdots,y^n);y^n>0\}$, for small $|y|$ there hold
\begin{equation}\label{vol_conf_Fermi_coor}
d\mu_{g_{x_0}}=(1+O(|y|^N))dy \mathrm{~~and~~} h_{g_{x_0}}=O(|y|^{N-1}) \mathrm{~~near~~} x_0.
\end{equation}
In this paper, $N=2d+2$ is enough for our use. Let $\Psi_{x_0}:B_{2\rho}^+\to M$ be a smooth map and set $x=\Psi_{x_0}(y)$, $\bar y=(y^1,\cdots,y^{n-1})$ are geodesic normal coordinates on $\pa M$ centered at $x_0$ and $\exp_{x_0}(-y^n\nu_{g_{x_0}}(x_0))\in M$ for small $y^n>0$. Then $y=(\bar y, y^n)$ is so-called $g_{x_0}$-Fermi coordinates around $x_0$. In particular, $(g_{x_0})_{nn}=1$ and $(g_{x_0})_{an}=0$ in $B_{2\rho}^+$. We write $(g_{x_0})_{ij}=\exp(h_{ij})$ and set
\begin{equation}\label{Def:H}
h_{ij}=\sum_{|\alpha|=1}^d \pa^\alpha h_{ij} y^\alpha+O(|y|^{d+1}):=H_{ij}+O(|y|^{d+1}).
\end{equation}
Denote by $\Omega_\rho=\Psi_{x_0}(B_\rho^+)$ the coordinate half-ball of radius $\rho$ under the Fermi coordinates around $x_0$. We adopt Einstein summation convention. 
\end{convention}

\section{The case of non-empty $\mathcal{Z}$ with PMT}\label{Sect3}

Indeed, our previous paper \cite{Chen-Sun} paves a way of giving a positive answer to the Han-Li conjecture with some natural assumptions imposed on manifolds. As mentioned earlier, the test function constructed in \cite{Chen-Sun} will play an important role in solving the Han-Li conjecture. For convenience, we collect some useful estimates therein and extend these estimates to non-positive $c$.

A simple but vital estimate for $W_\e$ is that for all $c \in \mathbb{R}$, there exists $C=C(T_c,n)>0$ such that
\begin{equation}\label{est:bubble_fcn}
C^{-1} \e^{\frac{n-2}{2}}(\e+|y|)^{2-n}\leq W_\e(y) \leq C \e^{\frac{n-2}{2}}(\e+|y|)^{2-n}, ~~y \in \mathbb{R}_+^n.
\end{equation}
For a simple proof, it is not hard to show \eqref{est:bubble_fcn} for $c \geq 0$. When $c<0$, i.e. $T_c>0$, it suffices to show
\begin{align*}
\e^2+|y-T_c\e \textbf{e}_n|^2\geq C^{-1} (\e+|y|)^2
\end{align*}
for some $C=C(T_c,n)>0$. To see this, given any $\delta\in (0,1)$, by Young's inequality we have
\begin{align*}
\e^2+|y-T_c\e \textbf{e}_n|^2=&\e^2(1+T_c^2)+|y|^2-2T_cy^n\e\\
\geq& \e^2[1+T_c^2(1-\delta^{-1})]+|\bar y|^2+(1-\delta)|y^n|^2.
\end{align*}
Then by choosing $\delta=2T_c^2/(1+2T_c^2)$, we obtain
\begin{equation*}
\e^2+|y-T_c\e \textbf{e}_n|^2\geq \min\left\{\tfrac{1}{2},\tfrac{1}{1+2T_c^2}\right\}(\e^2+|y|^2)\geq  \min\left\{\tfrac{1}{4},\tfrac{1}{2(1+2T_c^2)}\right\}(\e+|y|)^2
\end{equation*}
as desired.

We proceed to use the test function in \cite{Chen-Sun} and show that some required estimates therein remain true for non-positive $c$. Denote by $\chi(y)=\chi(|y|)$ a smooth cut-off function in $\overline{\mathbb{R}_+^n}$ with $\chi=1$ in $B_1^+$ and $\chi=0$ in $\mathbb{R}^n_+\backslash \overline{B_2^+}$. For any $\rho>0$, set $\chi_\rho(y)=\chi(|y|/\rho)$ for $y \in \mathbb{R}_+^n$. Let $G=G_{x_0}$ denote the Green's function of the conformal Laplacian of metric $g_{x_0}$ with pole $x_0\in \partial M$, coupled with a boundary condition, namely
\begin{align*}
\begin{cases}
\displaystyle -\frac{4(n-1)}{n-2}\Delta_{g_{x_0}}G_{x_0}+R_{g_{x_0}}G_{x_0}=0\,,&\mathrm{in~~}\:M\,,
\\
\displaystyle \frac{2}{n-2}\frac{\partial G_{x_0}}{\partial \nu_{g_{x_0}}} +h_{g_{x_0}}G_{x_0}=0\,,&\mathrm{on~~}\:\d M\backslash \{x_0\}\,.
\end{cases}
\end{align*}
We assume that $G$ is normalized such that $\lim_{y\to 0}G(\Psi_{x_0}(y))|y|^{n-2}=1$, where $\Psi_{x_0}$ is the conformal Fermi coordinates around $x_0$. In particular, $G$ is positive in $\overline M \setminus \{x_0\}$, see \cite[Appendix B]{Almaraz-Sun}.

As in \cite{Chen-Sun} we extend to define
\begin{equation}\label{test_fcn_Chen_Sun}
\ubar=[\chi_\rho(\U+\w)]\circ\Psi_{x_0}^{-1}+\epsilon^{\frac{n-2}{2}}(1-\chi_\rho)\circ\Psi_{x_0}^{-1}G
\end{equation}
for all $c \in \mathbb{R}$, where
\begin{equation}\label{eq:def:psi}
\w=\d_i\U V_i+\frac{n-2}{2n}\U \mathrm{div} V
\end{equation}
and $V$ is the vector field satisfying
\begin{align}\label{eq:V}
\begin{cases}
\sum\limits_{i=1}^{n}\pa_i\left[\U^{\crit}\Big(\chi_{\rho}H_{ij}-\d_iV_j-\d_jV_i+\frac{2}{n}(\mathrm{div} V)\delta_{ij}\Big)\right]=0,&\mathrm{~~in~~}\:\Rn,
\\
\d_nV_a=V_n=0,&\mathrm{~~on~~}\:\mathbb{R}^{n-1},
\end{cases}
\end{align}
where $1\leq i,j \leq n, 1\leq a\leq n-1$. Moreover, there holds
\begin{align}\label{est:V}
|\d^{\b}V(y)|\leq C(n,T_c, |\b|)\sum_{a,b=1}^{n-1}\sum_{|\a|=1}^{d}|\pa^\alpha h_{ab}|(\e+|y|)^{|\a|+1-|\b|}, \forall~ \beta.
\end{align}
Thus, using \eqref{est:V} and the expression \eqref{eq:bdry_bubble} of $\U$, in $B_{2\rho}^+$ we have
\begin{align}\label{est:w}
|\w(y)|\leq C(n,T_c)\epsilon^{\frac{n-2}{2}}\sum_{a,b=1}^{n-1}\sum_{|\alpha|=1}^d |\pa^\alpha h_{ab}|(\epsilon+|y|)^{|\alpha|+2-n}.
\end{align}

Under some natural conditions on manifolds,  we manage to show that the above $\ubar$ is a good candidate for \eqref{ieq:goal_less}. By choosing $\rho$ sufficiently small, $\ubar$ is non-negative by virtue of \eqref{est:w}.

We define symmetric trace-free 2-tensors $S$ and $T$ in $\overline \Rn$ by 
\begin{equation}\label{def:S:T}
S_{ij}=\d_iV_j+\d_jV_i-\frac{2}{n}\mathrm{div}V\delta_{ij}\quad\quad\mathrm{and}\quad\quad T=H-S\,.
\end{equation}

\begin{remark}
Though the following estimates \eqref{est:B_rho}, \eqref{energy_outside_ball}, \eqref{est:energy_mass} were proved in \cite{Chen-Sun} only for any positive $c$, thanks to estimate \eqref{est:bubble_fcn},  we can easily extend these ones to non-positive $c$ by following nearly the same lines in \cite{Chen-Sun}. 
\end{remark}

For the energy $E$ of this $\ubar$ in $\Omega_\rho=\Psi_{x_0}(B_\rho^+)$, denoted by $E[\ubar;\Omega_\rho]$ for brevity, it follows from \cite[Proposition 5.7]{Chen-Sun} that with a sufficiently small $\rho_0>0$, there holds
\begin{align}\label{est:B_rho}
&\int_{B^+_\rho}\left[\frac{4(n-1)}{n-2}|\nabla(\U+\w)|_{g_{x_0}}^2+R_{g_{x_0}}(\U+\w)^2\right]dy\no\\
&+2(n-1)\int_{D_\rho}h_{g_{x_0}}(\U+\w)^2d\sigma\no\\
\leq &4n(n-1)\int_{B^+_\rho}\U^{\frac{4}{n-2}}\left(\U^2+\frac{n+2}{n-2}\w^2\right)dy\no\\
&+\frac{4(n-1)}{n-2}\int_{\partial^+B_\rho^+}\partial_i\U\U\frac{y^i}{|y|}d\sigma+\int_{\pa^+B_\rho^+}(\U^2\pa_j h_{ij}-\pa_j\U^2 h_{ij})\frac{y^i}{|y|}d\sigma\no\\
&-4(n-1)T_c\int_{D_\rho}\U^{\frac{2}{n-2}}\Big(\U^2+2\U\w+\frac{n}{n-2}\w^2-\frac{n-2}{8(n-1)^2}\U^2S_{nn}^2\Big)d\sigma\no\\
&-\frac{1}{2}\lambda^* \sum_{a,b=1}^{n-1}\sum_{|\alpha|=1}^d |\pa^\alpha h_{ab}|^2\epsilon^{n-2}\int_{B_\rho^+}(\epsilon+|y|)^{2|\alpha|+2-2n}dy\no\\
&+C\sum_{a,b=1}^{n-1}\sum_{|\alpha|=1}^d |\pa^\alpha h_{ab}|\epsilon^{n-2} \rho^{|\alpha|+2-n}+C\epsilon^{n-2} \rho^{2d+4-n}
\end{align}
for $0<2\epsilon<\rho<\rho_0\leq 1$, where $\lambda^\ast, \rho_0, C$ are positive constants depending only on $n,~T_c,~g_0$.

For the energy $E$ of $\ubar$ outside the region $\Omega_\rho$, denoted by $E[\ubar; M\backslash\Omega_\rho]$ for simplicity, using \cite[estimate (5.48)]{Chen-Sun} we have
\begin{align}\label{energy_outside_ball}
&\int_{M\backslash\Omega_\rho}\left[\tfrac{4(n-1)}{n-2}|\nabla\ubar|_{g_{x_0}}^2+R_{g_{x_0}}\ubar^2\right]d\mu_{g_{x_0}}\no\\
&+2(n-1)\int_{\partial M\backslash\Omega_\rho}h_{g_{x_0}}\ubar^2 d\sigma_{g_{x_0}}\no\\
\leq &\frac{4(n-1)}{n-2}\int_{\partial^+B_\rho^+}\left[-\partial_i\U\U+\partial_j\U\U h_{ij}-\epsilon^{\frac{n-2}{2}}(\U\partial_iG-G\partial_i\U)\right]\frac{y^i}{|y|} d\sigma\no\\
&+C\sum_{a,b=1}^{n-1}\sum_{|\alpha|=1}^d |\pa^\alpha h_{ab}|\rho^{|\alpha|+2-n}\epsilon^{n-2}+C\sum_{a,b=1}^{n-1}\sum_{|\alpha|=1}^d |\pa^\alpha h_{ab}|^2\rho^{2|\alpha|+2-n}\epsilon^{n-2}\no\\
&+C\rho^{2d+4-n}|\log \rho|^2\epsilon^{n-2}+C\rho^{2-n}\epsilon^{n-1}.
\end{align}

Since $d\mu_{g_{x_0}}=\big(1+O(|y|^{2d+2})\big)dy$ and $d\sigma_{g_{x_0}}=\big(1+O(|y|^{2d+2})\big)d\sigma$, then under conformal Fermi coordinates around $x_0\in \pa M$, we combine the above two estimates to obtain
\begin{align*}
&E[\ubar]\\
\leq &4n(n-1)\int_{B^+_\rho}\U^{\frac{4}{n-2}}\left(\U^2+\frac{n+2}{n-2}\w^2\right)dy\\
&-4(n-1)T_c\int_{D_\rho}\U^{\frac{2}{n-2}}\Big(\U^2+2\U\w+\frac{n}{n-2}\w^2-\frac{n-2}{8(n-1)^2}\U^2S_{nn}^2\Big)d\sigma\\
&+\int_{\pa^+B_\rho^+}(\U^2\pa_j h_{ij}+\frac{n}{n-2}\pa_j\U^2 h_{ij})\frac{y^i}{|y|}d\sigma\\
&-\frac{4(n-1)}{n-2}\epsilon^{\frac{n-2}{2}}\int_{\partial^+ B_\rho^+}(\U\partial_iG-G\partial_i\U)\frac{y^i}{|y|}d\sigma\\
&-\frac{1}{2}\lambda^* \sum_{a,b=1}^{n-1}\sum_{|\alpha|=1}^d |\pa^\alpha h_{ab}|^2\epsilon^{n-2}\int_{B_\rho^+}(\epsilon+|y|)^{2|\alpha|+2-2n}dy\\
&+C\sum_{a,b=1}^{n-1}\sum_{|\alpha|=1}^d |\pa^\alpha h_{ab}|\rho^{|\alpha|+2-n}\epsilon^{n-2}+C\rho^{2d+4-n}|\log \rho|^2 \epsilon^{n-2}+C\epsilon^{n-1}\rho^{2-n}.
\end{align*}
Using \cite[estimate (5.50)]{Chen-Sun}
\begin{align}\label{est:energy_mass}
&\int_{\pa^+B_\rho^+}(\U^2\pa_j h_{ij}+\frac{n}{n-2}\pa_j \U^2 h_{ij})\frac{y^i}{|y|}d\sigma\no\\
&-\frac{4(n-1)}{n-2}\epsilon^{\frac{n-2}{2}}\int_{\partial^+ B_\rho^+}(\U\partial_iG-G\partial_i\U)\frac{y^i}{|y|}d\sigma\no\\
&\leq -\epsilon^{n-2}\mathcal{I}(x_0,\rho)+C\sum_{a,b=1}^{n-1}\sum_{|\alpha|=1}^d |\pa^\alpha h_{ab}|\rho^{|\alpha|+1-n}\epsilon^{n-1}+C\epsilon^{n-1}\rho^{1-n},
\end{align}
where $\mathcal{I}(x_0,\rho)$ is defined in \cite[P. 32]{Chen-Sun},
we emphasize that for all $c \in \mathbb{R}$, there holds
\begin{align}\label{est:energy_psi=0}
&E[\ubar]\no\\
\leq &4n(n-1)\int_{B^+_\rho}\U^{\frac{4}{n-2}}\left(\U^2+\frac{n+2}{n-2}\w^2\right)dy\no\\
&-4(n-1)T_c\int_{D_\rho}\U^{\frac{2}{n-2}}\Big(\U^2+2\U\w+\frac{n}{n-2}\w^2-\frac{n-2}{8(n-1)^2}\U^2S_{nn}^2\Big)d\sigma\no\\
&-\e^{n-2}\mathcal{I}(x_0,\rho)-\frac{1}{2}\lambda^* \sum_{a,b=1}^{n-1}\sum_{|\alpha|=1}^d |\pa^\alpha h_{ab}|^2\epsilon^{n-2}\int_{B_\rho^+}(\epsilon+|y|)^{2|\alpha|+2-2n}dy\no\\
&+C\sum_{a,b=1}^{n-1}\sum_{|\alpha|=1}^d |\pa^\alpha h_{ab}|\rho^{|\alpha|+2-n}\epsilon^{n-2}+C\rho^{2d+4-n}|\log \rho|^2 \epsilon^{n-2}+C\epsilon^{n-1}\rho^{2-n}.
\end{align}

By definition \eqref{main_funtional_I} of the functional $I$, we need the expansions of the volumes of $M$ and $\pa M$ under conformal Fermi coordinates around $x_0 \in \pa M$.
\begin{lemma}\label{lem:exp_vol_M}
 If $0<\e\ll\rho<\rho_0$ for some sufficiently small $\rho_0$, there holds
\begin{equation*}
\int_M \ubar^{\frac{2n}{n-2}}d\mu_{g_{x_0}}=A+A_2+O(\e^3)+O(\e^{n}\rho^{-n}),
\end{equation*}
where 
\begin{equation*}
A_2=\frac{n(n+2)^2}{n-2}\int_{B_{\rho}^+}\U^{\frac{4}{n-2}}\psi^2 dy=O(\e^2).
\end{equation*}
\end{lemma}
\begin{proof}
By \eqref{vol_conf_Fermi_coor} we obtain
\begin{align*}
\int_M \ubar^{\frac{2n}{n-2}}d\mu_{g_{x_0}}=&\int_{\Omega_{\rho}} ((\U+\w)\circ\Psi_{x_0}^{-1})^{\frac{2n}{n-2}}d\mu_{g_{x_0}}+O(\e^n\rho^{-n})\\
=&\int_{B_\rho^+}(\U+\psi)^{\frac{2n}{n-2}}dy+O(\e^n \rho^{N-n})+O(\e^n\rho^{-n}).
\end{align*}
Notice that $|\psi(y)|\leq (\e+|y|)\U(y)$ in $B_{2\rho}^+$. Observe that
\begin{align*}
&\int_{B_\rho^+}(\U+\psi)^{\frac{2n}{n-2}}dy\\
=&\int_{B_\rho^+}\U^{\frac{2n}{n-2}}dy+\frac{2n}{n-2}\int_{B_\rho^+}\U^{\frac{n+2}{n-2}}\psi dy+\frac{n(n+2)}{(n-2)^2}\int_{B_\rho^+} \U^{\frac{4}{n-2}}\psi^2dy+O(\e^3)\\
:=&A_0+A_1+A_2+O(\e^3).
\end{align*}
It is easy to show
\begin{align}\label{volume_I_1}
A_0=A+O(\e^n\rho^{-n}).
\end{align}
By definition \eqref{eq:def:psi} of $\psi$, we have
$$\frac{2n}{n-2}\U^{\frac{n+2}{n-2}}\psi=\mathrm{div}(\U^{\frac{2n}{n-2}}V).$$
Since $V_n=0$ on $D_\rho$, an integration by parts together with \eqref{est:V} gives
\begin{align*}
A_1=&\int_{B_\rho^+}\mathrm{div}(\U^{\frac{2n}{n-2}}V) dy=\int_{\partial^+ B_\rho^+}\U^{\frac{2n}{n-2}}V_i\frac{y^i}{|y|}d\sigma=O(\e^{n}\rho^{1-n}).
\end{align*}
By collecting all the above estimates, the desired assertion follows.
\end{proof}
\begin{lemma} \label{lem:exp_vol_bdry}
If $0<\e\ll\rho<\rho_0$ for some sufficiently small $\rho_0$, there holds
\begin{align*}
\int_{\pa M}\ubar^{\frac{2(n-1)}{n-2}}d\sigma_{g_{x_0}}=&B+B_1+B_2+O(\e^3)+O(\e^{n-1}\rho^{1-n}),
\end{align*}
where
\begin{align*}
B_1=\frac{2(n-1)}{n-2}\int_{D_{\rho}}\U^{\frac{n}{n-2}}\psi d\sigma,\quad B_2=\frac{n(n-1)}{(n-2)^2}\int_{D_\rho}\U^{\frac{2}{n-2}}\psi^2d\sigma.
\end{align*}
\end{lemma}
\begin{proof} By \eqref{vol_conf_Fermi_coor} we have
\begin{align*}
\int_{\partial M}\ubar^{\frac{2(n-1)}{n-2}}d\sigma_{g_{x_0}}=&\int_{\partial\Omega_\rho\cap\partial M}((\U+\psi)\circ\Psi_{x_0}^{-1})^{\frac{2(n-1)}{n-2}}d\sigma_{g_0}+O(\e^{n-1}\rho^{1-n})\\
=&\int_{D_\rho }(\U+\psi)^{\frac{2(n-1)}{n-2}}d\sigma+O(\e^{n-1}\rho^{N-n+1})+O(\e^{n-1}\rho^{1-n}).
\end{align*}
Notice that
\begin{align*}
&\int_{D_\rho}(\U+\psi)^{\frac{2(n-1)}{n-2}}d\sigma\\
=&\int_{D_\rho}\U^{\frac{2(n-1)}{n-2}}d\sigma+\frac{2(n-1)}{n-2}\int_{D_\rho}\U^{\frac{n}{n-2}}\psi d\sigma\\
&+\frac{(n-1)n}{(n-2)^2}\int_{D_\rho}\U^{\frac{2}{n-2}}\psi^2d\sigma+O(\e^3)\\
=&B+B_1+B_2+O(\e^3)+O(\e^{n-1}\rho^{1-n}).
\end{align*}
Hence we combine these two estimates to obtain the desired estimate. 
\end{proof}

We will use a notion of a {\it mass} associated to manifolds with boundary.
\begin{definition}\label{def:asym}
Let $(N, g)$ be a Riemannian manifold with a non-compact boundary $\pa N$. 
We say that $N$ is {\it{asymptotically flat}} with order $p>0$, if there exist a compact set $N_0\subset N$ and a diffeomorphism $F: N\backslash N_0\to \Rn\backslash \overline{B^+_1(0)}$ such that, in the coordinate chart defined by $F$ (called {\it  asymptotic coordinates} of $N$), there holds
$$
|g_{ij}(y)-\delta_{ij}|+|y||\pa_k g_{ij}(y)|+|y|^2|\pa^2_{kl}g_{ij}(y)|=O(|y|^{-p}),\mathrm{~~as~~}|y|\to\infty,
$$
where $i,j,k,l=1,\cdots,n, B_1^+(0)=B_1(0)\cap \mathbb{R}_+^n$.
\end{definition}

Provided that the following limit
\begin{align*}
&m(g):=\\
&\lim_{R\to\infty}\left[
\int\limits_{\{y\in\Rn;\, |y|=R\}}\sum_{i,j=1}^{n}(g_{ij,j}-g_{jj,i})\frac{y^i}{|y|}\,d\sigma
+\int\limits_{\{y\in\mathbb{R}^{n-1};\, |y|=R\}}\sum_{a=1}^{n-1}g_{na}\frac{y^a}{|y|}\,d\sigma\right]
\end{align*}
exists, we call it the {\it{mass}} of $(N, g)$.  The following positive mass type conjecture has been verified by  Almaraz-Barbosa-de Lima \cite{almaraz-barbosa-lima}, provided that $3\leq n\leq 7$ or $n\geq 8$ and $N$ is spin. 
\begin{conjecture}[\textbf{Positive mass with a non-compact boundary}]
If $(N,g)$ is asymptotically flat with order $p>(n-2)/2$ and $R_{g}\geq 0, h_g\geq 0$, then $m(g)\geq 0$, equality holds if and only if $(N,g)$ is isometric to $(\mathbb{R}_+^n,g_{\mathbb{R}^n})$.
\end{conjecture}

For readers' convenience, we present a proof of the following elementary result, though it is maybe well-known to the experts in this field.
\begin{proposition}\label{prop:vanishing_order_metric}
For $n\geq 3$ and $x_0 \in \mathcal{Z}$, let $g_{x_0} \in [g_0]$ be the metric induced by the conformal Fermi coordinates around $x_0$, then there holds $(g_{x_0})_{ij}=\delta_{ij}+O(|y|^{d+1})$ near $x_0$.
\end{proposition}
\begin{proof}
When $n=3$, then $d=0$ and the assertion is trivial. Now without loss of generality, we assume $d \geq 1$. Under conformal $g_{x_0}$-Fermi coordinates near $x_0 \in \pa M$, it follows from F. Marques \cite[Proposition 3.1]{marques1} that $\det g_{x_0}=1+O(|y|^{2d+2})$ in $B_{2\rho}^+$. If we write $(g_{x_0})_{ij}=\exp(h_{ij})$ and 
$$h_{ij}=\sum_{|\alpha|=1}^d \pa^\alpha h_{ij} y^\alpha+O(|y|^{d+1}):=H_{ij}+O(|y|^{d+1}),$$
then $(g_{x_0})^{ij}h_{ij}=O(|y|^{2d+2})$ in $B_{2\rho}^+$. Moreover, the mean curvature $h_{g_{x_0}}$ satisfies
\begin{align*}
h_{g_{x_0}}=-\frac{1}{2(n-1)}(g_{x_0})^{ij}\partial_n(g_{x_0})_{ij}=-\frac{1}{2(n-1)}\partial_n(\log\det(g_{x_0}))=O(|x|^{2d+1}).
\end{align*}
 Indeed we can prove a little stronger result.
\begin{claim}
Suppose that $|W_{g_{x_0}}|_{g_{x_0}}=O(|y|^{k-1})$, $|\mathring\pi_{g_{x_0}}|_{g_{x_0}}=O(|y|^k)$ and $h_{ij}=H_{ij}+O(|y|^{k+1})$ near $x_0$, where $H_{ij}=\sum_{|\alpha|=1}^k\partial^\alpha h_{ij}(0)y^\alpha$ and $1\leq k\leq d$, then $H_{ij}=0$ near $x_0$.
\end{claim}
We prove it by induction on $k$. 

Suppose that $k=1$. Since $\pa_n (g_{x_0})_{ab}=-\frac{1}{2}(\pi_{g_{x_0}})_{ab}=O(|y|)$ near $x_0 \in \pa M$, then $\pa_n (g_{x_0})_{ab}(0)=0$, which implies $\partial_n h_{ab}(0)=0$. On the other hand, it follows from conformal Fermi coordinates that $\partial_c h_{ab}(0)=0$ for any $1\leq a,b,c\leq n-1$. Thus we obtain $H_{ab}=0$ near $x_0$. One consequence of Fermi coordinates is $H_{in}=0$ for $1\leq i\leq n$. Thus $k=1$ is proved.

 Assume that the claim holds for $k \geq 1$ and consider the case $k+1$. Firstly from induction assumption, we obtain
$$h_{ij}=H_{ij}+O(|y|^{k+2}),\quad \text{where }H_{ij}=\sum_{|\alpha|=k+1}\partial^\alpha h_{ij}(0)y^\alpha.$$
Notice that
\begin{equation}\label{eq:vanish_bdry_con}
\pa_n H_{ij}=\sum_{|\bar \alpha|=k}\pa^{(\bar \alpha,n)}h_{ij}(0)y^{\bar \alpha},\quad \mathrm{~~on~~} \{y^n=0\},
\end{equation}
where $\bar \alpha$ denotes a multi-index with components consisting of $\{1,2,\cdots,n-1\}$. Also we have $\partial^{\bar \alpha}\pa_n (g_{x_0})_{ab}=-\frac{1}{2}\partial^{\bar \alpha}(\pi_{g_{x_0}})_{ab}=O(|y|)$ for all $|\bar \alpha|=k$, by virtue of induction assumption for $k+1$ that $|\mathring{\pi}_{g_{x_0}}|_{g_{x_0}}=O(|y|^{k+1})$ and the mean curvature $h_{g_{x_0}}=O(|x|^{2d+1})$. Then $\partial_nH_{ab}=0$ in a small neighborhood on $\pa M$ of $x_0$.  Since we are using Fermi coordinates, then all $\partial_nH_{ij}=0$ on $\partial M $ near $x_0$.

As in \cite{Brendle2} (see also \cite{almaraz5}), in $B_{2\rho}^+$ we define
$$A_{ik}=\pa_i \pa_m H_{mk}+\pa_m\pa_k H_{im}-\Delta H_{ik}-\frac{1}{n-1}\pa_m \pa_p H_{mp}\delta_{ik}$$
and
\begin{align}
Z_{ijkl}=&\pa_i\pa_k H_{jl}+\pa_j\pa_l H_{ik}-\pa_i\pa_l H_{jk}-\pa_j\pa_k H_{il}\no\\
&+\frac{1}{n-2}(A_{jl}\delta_{ik}+A_{ik}\delta_{jl}-A_{jk}\delta_{il}-A_{il}\delta_{jk}).\label{def:Z}
\end{align}
Through direct computations (see also \cite[formula (11) in the proof of Lemma 2.1]{Ambrosetti-Malchiodi}), we obtain the following expansion of the Weyl tensor $W_{ijkl}$ of $g_{x_0}$
\begin{equation}\label{expan_Weyl_tensor}
W_{ijkl}=-2Z_{ijkl}+O(|\pa (h\partial h)|^2)+O(|y|^{k}).
\end{equation}
By our assumption, it follows that $|W_{g_{x_0}}|_{g_{x_0}}=O(|y|^{k})$ and $|\partial (h\partial h)|^2=O(|y|^{2k})$.  This implies $Z_{ijkl}= 0$ near $x_0$, because \eqref{def:Z} shows that $Z_{ijkl}$ are only homogeneous polynomials of degree $k-1$. Since we have already shown that $\partial_n H_{ij}=0$ on $\{y^n=0\}$, then one can deduce that $H=0$ from \cite[Proposition 2.3]{Brendle-Chen}. Hence the induction step is complete.

After finitely many steps, if $k=d$, we obtain the desired conclusion.
\end{proof}

\begin{proposition}\label{prop:z=nonemptyset}
Suppose that $m(\bar g_{x_0})>0$ for $x_0 \in \mathcal{Z}$, where $\bar g_{x_0}=G_{x_0}^{4/(n-2)}g_{x_0}$ is the metric on $\overline{M} \setminus\{x_0\}$ (see \cite[Proposition 5.14]{Chen-Sun}). Then $\ubar$ satisfies \eqref{ieq:goal_less} when $\e$ is sufficiently small.
\end{proposition}
\begin{proof} We obtain that $(g_{x_0})_{ij}=\delta_{ij}+O(|y|^{d+1})$ near $x_0$, by the assumption and Proposition \ref{prop:vanishing_order_metric}. Then \eqref{est:V} implies that $V=\psi\equiv 0$ in $B_{2\rho}^+$, whence $\psi=S_{nn}=0$ in $B_{2\rho}^+$ by definitions of $\psi$ and $S_{nn}$. By \eqref{vol_conf_Fermi_coor} we have
\begin{align}\label{vol_M_psi=0}
\int_M \ubar^{\frac{2n}{n-2}}d\mu_{g_{x_0}}=&\int_{\Omega_{\rho}} (\U\circ\Psi_{x_0}^{-1})^{\frac{2n}{n-2}}d\mu_{g_{x_0}}+O(\e^n\rho^{-n})\no\\
=&\int_{B_\rho^+}\U^{\frac{2n}{n-2}}dy+O(\e^{n}\rho^{2d+2-n})+O(\e^n\rho^{-n})\no\\
=&A+O(\e^n\rho^{-n}).
\end{align}
Similarly we have
\begin{align}\label{vol_bdry_psi=0}
\int_{\partial M}\ubar^{\frac{2(n-1)}{n-2}}d\sigma_{g_{x_0}}=&\int_{\partial\Omega_\rho\cap\partial M}(\U\circ\Psi_{x_0}^{-1})^{\frac{2(n-1)}{n-2}}d\sigma_{g_{x_0}}+O(\e^{n-1}\rho^{1-n})\no\\
=&\int_{D_\rho }\U^{\frac{2(n-1)}{n-2}}d\sigma+O(\e^{n-1} \rho^{2d+3-n})+O(\e^{n-1}\rho^{1-n})\no\\
=&B+O(\e^{n-1}\rho^{1-n}).
\end{align}
Moreover, since $m(\bar g_{x_0})>0$, we obtain from \cite[Proposition 4.3]{Brendle-Chen} that $\mathcal{I}(x_0,\rho)>\tilde{C}>0$ for sufficiently small $\rho>0$. From this and Proposition \ref{prop:vanishing_order_metric}, \eqref{est:energy_psi=0} together with $\psi=S_{nn}=0$ implies
\begin{align}\label{E_psi=0}
\frac{n-2}{4(n-1)}E[\ubar]\leq& n(n-2)A+cB-\frac{\tilde{C}}{2}\e^{n-2}
=\int_{\mathbb{R}^n_+}|\nabla \U|^2dy-\frac{\tilde{C}}{2}\e^{n-2},
\end{align}
where the last identity follows from \eqref{energy:bubble2}. 

Denote by $t_\ast>0$ the unique maximum point of \eqref{eq:constraint_max_t} with $u=\ubar$. Thus, by \eqref{eq:qudratic_root_t} together with \eqref{vol_M_psi=0}, \eqref{vol_bdry_psi=0}, \eqref{E_psi=0} and \eqref{energy:bubble2}, we conclude that $t_\ast \in (0,1)$ when $\e$ is sufficiently small.
Therefore, for $0<\e\ll\rho< \rho_0$ with small enough $\rho_0$, we obtain 
\begin{align*}
I[t_\ast \ubar]=&\frac{1}{n-1}E[\ubar]t_\ast^2+4(n-2)\int_M(\ubar)^{\frac{2n}{n-2}}d\mu_{g_{x_0}}t_\ast^{\frac{2n}{n-2}}\\
<& \frac{4}{n-2}\int_{\mathbb{R}^n_+}|\nabla W(y)|^2dy+4(n-2)\int_{\mathbb{R}^n_+}W^{\frac{2n}{n-2}}dy=S_c,
\end{align*}
where the last inequality follows from \eqref{vol_M_psi=0}, \eqref{E_psi=0}, $Y(M,\pa M,[g_0])>0$ and $t_\ast \in (0,1)$, namely, this $\ubar$ satisfies \eqref{ieq:goal_less}. 
\end{proof}

It follows from \cite{han-li2} that if $n\geq 5$ and $\pa M$ admits a non-umbilic point, then the Han-Li conjecture is true. From this, a direct consequence of Proposition \ref{prop:z=nonemptyset} is the following
\begin{theorem}\label{thm:Han-Li_lower_dims}
The Han-Li conjecture is true, provided that any of the following assumptions is fulfilled:
\begin{enumerate}
\item [(i)] $n=3$;
\item [(ii)] $n=4,5$ and $\partial M$ has an umbilic point;
\item [(iii)] $n=6,7$, the boundary is umbilic and the Weyl tensor of $M$ vanishes at some boundary point;
\item [(iv)] $\mathcal{Z}$ is non-empty and $M$ is spin or $M$ is locally conformally flat with umbilic boundary.
\end{enumerate}
\end{theorem}
\begin{proof} Based on Proposition \ref{prop:z=nonemptyset}, it reduces to showing that there exists $x_0 \in \mathcal{Z}$ such that $m(\bar g_{x_0})>0$. We will verify these term by term.

(i) If $n=3$, then $d=0$ and  $\mathcal{Z}=\partial M$.

(ii) If $n=4,5$ and let $x_0\in \pa M$ be an umbilic point, then $d=1$ and
$$\limsup_{x\to x_0}d_{g_0}(x,x_0)^{2-d}|W_{g_0}(x)|_{g_0}=\limsup_{x\to x_0}d_{g_0}(x,x_0)^{1-d}|\mathring{\pi}_{g_0}(x)|_{g_0}=0.$$
Thus, $x_0\in \mathcal{Z}$. 

(iii) If $n=6,7$, then $d=2$. Thanks to Han-Li's result \cite{han-li2}, we can assume the boundary is umbilic, i.e. $\mathring \pi_{g_0}=0$ on $\pa M$. Combining this and the assumption that the Weyl tensor of $M$ is zero at  some $x_0 \in \pa M$, we have $x_0\in \mathcal{Z}$.
 
 (iv) In this case, $\mathcal{Z}\neq \emptyset$. Also the positive mass type theorem has been verified in \cite[Theorem 1.3]{almaraz-barbosa-lima} and \cite{escobar4}, respectively.
\end{proof}

 Again from \cite{han-li2}  together with Theorem \ref{thm:Han-Li_lower_dims}, the remaining cases in lower dimension $3 \leq n \leq 7$ are
\begin{enumerate}
\item[$\bullet$] $n=4$ and $\partial M$ admits at least one non-umbilic point;
\item[$\bullet$] $n=6,7$, $\pa M$ is umbilic and the Weyl tensor $W_{g_0}$ of $M$ does not vanish everywhere on $\pa M$.
\end{enumerate}

\section{Remaining cases in dimensions four and six}\label{Sect4}

We first assume that $n=4$ and $\pa M$ admits at least a non-umbilic point $x_0$, then it follows from F. Marques \cite[Lemma 2.2]{marques1} that there exists $g_{x_0} \in [g_0]$  such that under $g_{x_0}$-Fermi coordinates around $x_0$, it has the following expansion near $x_0$:
\begin{align}\label{conf_Fermi_expansion}
(g_{x_0})_{ab}=\delta_{ab}-2\pi_{ab}y^n+O(|y|^2).
\end{align}
Then \eqref{conf_Fermi_expansion} implies 
\begin{align}\label{exp:h:n=4}
h_{ab}=-2\pi_{ab}y^n+O(|y|^2).
\end{align}

Define a test function as 
\begin{equation}\label{test_fcn_n=4}
\ubar(x)=[\chi_\rho(\U+\w)\circ\Psi^{-1}_{x_0}](x),
\end{equation}
where $\w$ is defined in \eqref{eq:def:psi}. Since $d=1$ when $n=4$, in $B_{2\rho}^+$ \eqref{est:w} gives 
\begin{align}\label{est:w:n=4}
|\w(y)|\leq C\sum_{a,b=1}^{n-1}|\pi_{ab}|(\e+|y|)\U(y).
\end{align}
Although the test function in this case only has the local feature, one can apply almost identical argument as Lemmas \ref{lem:exp_vol_M} and \ref{lem:exp_vol_bdry} to get the following two lemmas.
\begin{lemma}\label{lem:volume_in_n=4} If $0<\e\ll\rho<\rho_0$ for some sufficiently small $\rho_0$, there holds
\begin{align*}
\int_M \ubar^{\frac{2n}{n-2}}d\mu_{g_{x_0}}=&A+O(\e^{2}).
\end{align*}
\end{lemma}

\begin{lemma}\label{lem:volume_bd_n=4} 
If $0<\e\ll\rho<\rho_0$ for some sufficiently small $\rho_0$, there holds
\begin{align*}
\int_{\pa M}\ubar^{\frac{2(n-1)}{n-2}}d\sigma_{g_{x_0}}=&B+B_1+O(\e^2),
\end{align*}
where
\begin{align*}
B_1=&\frac{2(n-1)}{n-2}\int_{D_\rho}\U^{\frac{n}{n-2}}\psi d\sigma
\end{align*}
and then $|B_1|\leq C(n,T_c)\e$.
\end{lemma}
A direct calculation shows
\begin{align}\label{eq:gradient_U}
|\nabla \U|^2=\e^{n-2}(n-2)^2\frac{|\bar y|^2+|y^n-T_c\e|^2}{(\e^2+|y-T_c\e\textbf{e}_n|^2)^n}.
\end{align}
Notice that 
\begin{align*}
|\nabla \ubar |^2_{g_{x_0}}\leq C|\nabla \ubar|^2\leq C\left[|\nabla\chi_\rho|^2(\U+\w)^2+\chi^2_\rho|\nabla(\U+\w)|^2\right],
\end{align*}
it yields
\begin{align*}
\int_{B_{2\rho}^+\backslash B_{\rho}^+}|\nabla\chi_\rho|^2(\U+\w)^2dy\leq& C\rho^{-2}\int_{B_{2\rho}^+\backslash B_{\rho}^+}\U^2dy\leq C\e^{n-2}\rho^{2-n},\\
\int_{B_{2\rho}^+\backslash B_{\rho}^+}\chi_\rho^2|\nabla(\U+\w)|^2dy\leq& C \int_{B_{2\rho}^+\backslash B_{\rho}^+}|\nabla \U|^2 dy \leq C\e^{n-2}\rho^{2-n}. 
\end{align*}
Similarly, we have
\begin{align*}
\int_{\Omega_{2\rho}\backslash \Omega_{\rho}}R_{g_{x_0}}\ubar^2d\mu_{g_{x_0}}\leq& C\int_{B_{2\rho}^+\backslash B_{\rho}^+}(\U+\w)^2dy\leq C\e^{n-2}\rho^{4-n},\\
\int_{\Psi_{x_0}(D_{2\rho}\backslash D_{\rho})}h_{g_{x_0}}\ubar^2d\sigma_{g_{x_0}}\leq& C\int_{D_{2\rho}\backslash D_{\rho}}|\bar y|^{N-1}(\U+\w)^2d\sigma\\
\leq& C\e^{n-2}\rho^{2+N-n}.
\end{align*}
Hence we obtain
\begin{align}\label{est:energy_outside_n=4}
E[\ubar;M\backslash \Omega_\rho]\leq C\e^{n-2}\rho^{2-n}.
\end{align}

Next we turn to estimate $E[\ubar; \Omega_\rho]$. By \eqref{exp:h:n=4} and \eqref{est:V}, estimate \eqref{est:B_rho} with $n=4$ actually implies
\begin{align}\label{est:energy_inside_n=4}
\frac{n-2}{4(n-1)}E[\ubar;\Omega_\rho]\leq& n(n-2)A+c\left(B+2\int_{D_\rho}\U^{\frac{n}{n-2}}\w d\sigma\right)+O(\e^2\rho^{2-n})\no\\
&-\frac{1}{2}|\pi_{g_{x_0}}|_{g_{x_0}}^2\e^2\int_{B_\rho^+}(\e+|y|)^{4-2n}dy,
\end{align}
where $|\pi_{g_{x_0}}|_{g_{x_0}}^2=\sum_{a,b=1}^{n-1}\pi^2_{ab}>0$ at $x_0$ by \eqref{exp:h:n=4} and the fact that $h_{g_{x_0}}=0 $ at this non-umbilic point $x_0$. Notice that 
$$\int_{B_\rho^+}(\e+|y|)^{4-2n}dy=O(\log \frac{\rho}{\e}).$$
Therefore, combining \eqref{est:energy_outside_n=4} and \eqref{est:energy_inside_n=4}, we can estimate 
\begin{align}\label{est:energy_n=4}
\frac{n-2}{4(n-1)}E[\ubar]\leq&n(n-2)A+c\left(B+2\int_{D_\rho}\U^{\frac{n}{n-2}}\w d\sigma\right)\no\\
&-\frac{1}{2}\lambda^\ast |\pi_{g_{x_0}}|_{g_{x_0}}^2\e^2\int_{B_\rho^+}(\e+|y|)^{-4}dy+O(\e^2\rho^{-2}).
\end{align}
Next we intend to apply Lemmas \ref{lem:volume_in_n=4} and \ref{lem:volume_bd_n=4} and equation \eqref{est:energy_n=4} to show that $\ubar$ satisfies \eqref{ieq:goal_less}. 

Applying \eqref{est:energy_n=4}, Lemmas \ref{lem:volume_in_n=4} and \ref{lem:volume_bd_n=4} to give
\begin{align*}
\mathcal{A}=&n(n-2)A+O(\e^2),\\
\mathcal{B}=&cB+cB_1+O(\e^2),\\
\mathcal{E}\leq&n(n-2)A+cB+\tfrac{n-2}{n-1}cB_1-\frac{1}{2}\lambda^\ast |\pi_{g_{x_0}}|_{g_{x_0}}^2\e^2\int_{B_\rho^+}(\e+|y|)^{-4}dy+O(\e^2\rho^{-2})
\end{align*}
and $B_1=O(\e)$, we obtain
\begin{align*}
\mathcal{B}^2+4\mathcal{E}\mathcal{A}
\leq&(cB+2n(n-2)A)^2+\left[2cB+\frac{4n(n-2)^2}{n-1}A\right]cB_1\\
&-2n(n-2)A\lambda^\ast\e^2|\pi_{g_{x_0}}|_{g_{x_0}}^2\int_{B_\rho^+}(\e+|y|)^{-4}dy+O(\e^2\rho^{-2}).
\end{align*}
Notice that $2n(n-2)A+cB>0$ for all $c \in \mathbb{R}$ because of \eqref{energy:bubble2}. Thus, \eqref{eq:qudratic_root_t} implies
\begin{align}\label{est:t_n=4}
t^{\frac{2}{n-2}}_*\leq&1-\frac{c}{(n-1)(2n(n-2)A+cB)}B_1\no\\
&\quad-\frac{\lambda^\ast |\pi_{g_{x_0}}|_{g_{x_0}}^2}{2(2n(n-2)A+cB)}\e^2\int_{B_\rho^+}(\e+|y|)^{-4}dy+O(\e^2\rho^{-2})\no\\
:=&1+\frac{2}{n-2}\tilde{B}_1-C^*\e^2\log(\rho/\e)+O(\e^2\rho^{-2}),
\end{align}
where $C^\ast$ is a positive constant depending on $c,\lambda^\ast,|\pi_{g_{x_0}}|_{g_{x_0}}^2$ and
$$\tilde{B}_1=\frac{n-2}{2(n-1)}\frac{-c}{2n(n-2)A+cB}B_1=O(\e).$$
Plugging \eqref{est:t_n=4}, \eqref{est:energy_n=4} and \eqref{eq:S_c_n=4} into \eqref{max_I_at_tu} and using \eqref{eq:S_c_n=4} and the assumption that $Y(M,\pa M,[g_0])>0$, we conclude that
\begin{align*}
&\max_{t \in[0,\infty)}I[t\ubar]\\
=&\frac{1}{n-1}E[t_*\ubar]+4(n-2)\int_M(t_*\ubar)^{\frac{2n}{n-2}}d\mu_{g_{x_0}}\\
\leq&S_c+\frac{4c}{n-1}B_1+8\tilde{B}_1(nA+\frac{1}{n-2}cB)+8nA\tilde{B}_1-\tilde{C}\e^2\log(\rho/\e)+O(\e^2\rho^{-2})\\
=&S_c-\tilde{C}\e^2\log(\rho/\e)+O(\e^2\rho^{-2})\\
<& S_c,
\end{align*}
for some $\tilde C=\tilde C(n,T_c, |\pi_{g_{x_0}}|_{g_{x_0}}^2)>0$, when $0<\e\ll\rho<\rho_0$ for some sufficiently small $\rho_0$, where the second identity follows from definitions of $B_1$ and $\tilde B_1$.

In the following, we go to dimensions $n=6,7$. Recall that when $n=6,7$, the remaining case is that the Weyl tensor of $M$ is non-zero everywhere on the umbilic boundary $\pa M$. Indeed we can achieve our goal when relaxing the assumption a little: 
Assume that $n\geq 6$, $\pa M$ is umbilic and the Weyl tensor $W_{g_0}$ of $M$ is non-zero at some $x_0 \in \pa M$, then  at $x_0$ there hold
\begin{equation}\label{est:h_higher_dim}
\pa^\alpha h_{ab}=0 \mathrm{~~for~~all~~} 0\leq |\alpha| \leq 1 \mathrm{~~and~~}\sum_{a,b=1}^{n-1}\sum_{|\alpha|=2}|\pa^\alpha h_{ab}|^2>0.
\end{equation} 
Here the above second assertion in \eqref{est:h_higher_dim} follows from a contradiction argument by using \eqref{expan_Weyl_tensor} and the assumption that $W_{g_0}(x_0)\not =0$.
Thus, it follows from \eqref{est:w} and  \eqref{est:h_higher_dim} that  in $B_{2\rho}^+$ there holds 
\begin{equation}\label{est:psi_higher_dim}
|\psi(y)|\leq C(\e+|y|)^2\U(y).
\end{equation}

In dimension six, a local test function is also enough to our use and the argument can be similarly done as the one in dimension four. Suppose that $n=6$, we still adopt the same test function $\ubar$ in \eqref{test_fcn_n=4} except for replacing $n=4$ by $n=6$. Similarly, with the above refined estimate \eqref{est:psi_higher_dim}, we obtain 
\begin{align}
\int_M \ubar^{\frac{2n}{n-2}}d\mu_{g_{x_0}}=&A+O(\e^{4})+O(\e^6\rho^{-6}),\label{est:volume_in_n=6}\\
\int_{\pa M}\ubar^{\frac{2(n-1)}{n-2}}d\sigma_{g_{x_0}}=&B+B_1+O(\e^4)+O(\e^6\rho^{-6}),\label{est:volume_bd_n=6}
\end{align}
where
\begin{align*}
B_1=&\frac{2(n-1)}{n-2}\int_{D_\rho}\U^{\frac{n}{n-2}}\psi d\sigma
\end{align*}
and then $|B_1|\leq C(n,T_c)\e^2$. Similar to \eqref{est:energy_n=4}, by using \eqref{est:energy_outside_n=4} together with $n=6$, \eqref{est:volume_in_n=6}, \eqref{est:volume_bd_n=6} and  \eqref{est:B_rho},  we estimate
\begin{align}\label{est:energy_n=6}
&\frac{n-2}{4(n-1)}E[\ubar]\no\\
\leq&n(n-2)A+c\left(B+2\int_{D_\rho}\U^{\frac{n}{n-2}}\w d\sigma\right)\no\\
&-\frac{1}{2}\lambda^\ast\sum_{|\alpha|=2}\sum_{a,b=1}^{n-1}|\partial^\alpha h_{ab}|^2\e^4\int_{B_\rho^+}(\e+|y|)^{-6}dy+O(\e^4\rho^{-4}).
\end{align}
Notice that
$$\int_{B_\rho^+}(\e+|y|)^{-6}dy=O(\log\frac{\rho}{\e}).$$
With these estimates and \eqref{eq:constraint_max_t} for this $\ubar$, one can proceed as the case $n=4$ and find that all $\e^2$-terms involved in $I[t_\ast \ubar]$ are cancelled out, eventually prove that $\ubar$ satisfies \eqref{ieq:goal_less}.

\section{$n \geq 7$, umbilic boundary and non-zero Weyl tensor at a boundary point }\label{Sect5}

In this section, we assume that $n\geq 7$, $\pa M$ is umbilic and the Weyl tensor of $M$ is non-zero at some $x_0 \in \pa M$. In Subsection \ref{Subsect5.1}, we still adopt the test function $\ubar$ defined in \eqref{test_fcn_Chen_Sun} to prove the Han-Li conjecture in addition that the constant $c$ is not less than a negative dimension constant. In Subsection \ref{Subsect5.2}, we explicitly construct a local test function to prove the Han-Li conjecture for all non-positive constant $c$.

\subsection{Positive constant boundary mean curvature}\label{Subsect5.1}

In this subsection, we still adopt the test function as defined in \eqref{test_fcn_Chen_Sun}. In contrast with those cases in dimensions four and six, due to the loss of the $\log|\e|$-term from the first correction term in the following estimate of $E$, as well as of $I$, the situation becomes more complicated. 
We start with some elementary calculations and temporarily admit the following expansions:
\begin{align}
\mathcal{A}=&n(n-2)A+n(n-2)\tilde{A}+O(\e^6),\label{eq:A}\\
\mathcal{B}=&cB+c\tilde{B}+O(\e^6),\label{eq:B}\\
\mathcal{E}=&n(n-2)A+cB+\tilde{E}+O(\e^5|\log\e|),\label{eq:E}
\end{align}
where $|\tilde A|\leq C\e^4$ and $|\tilde B| \leq C \e^2$ and $\tilde E\leq C\e^2$.

A direct computation together with \eqref{energy:bubble2} yields 
\begin{align*}
\mathcal{B}^2+4\mathcal{E}\mathcal{A}=&(2n(n-2)A+cB)^2+2c^2B\tilde B+4n(n-2)A\tilde E\\
&+c^2\tilde B^2+4n(n-2)\tilde A(n(n-2)A+cB)+o(\e^4)
\end{align*}
and 
\begin{align*}
&-\mathcal{B}+\sqrt{\mathcal{B}^2+4\mathcal{E}\mathcal{A}}\\
=&2n(n-2)A+2n(n-2)\frac{A\tilde E-cA\tilde B+\frac{c^2\tilde{B}^2}{4n(n-2)}+[n(n-2)A+cB]\tilde A}{2n(n-2)A+cB}\\
&-\frac{1}{2}\frac{(c^2B\tilde B+2n(n-2)A\tilde E)^2}{(2n(n-2)A+cB)^3}+o(\e^4).
\end{align*}
It follows from the above estimates and $\eqref{eq:qudratic_root_t}$ that
\begin{align*}
t_*^{\frac{2}{n-2}}=&1+\frac{\tilde E-c\tilde B}{2n(n-2)A+cB}+\frac{\frac{c^2\tilde B^2}{4n(n-2)A}-n(n-2)\tilde A}{2n(n-2)A+cB}\\
&-\frac{(c^2B\tilde B+2n(n-2)A\tilde E)^2}{4n(n-2)A(2n(n-2)A+cB)^3}+o(\e^4)\\
:=&1+\frac{2}{n-2}T_*+o(\e^4),
\end{align*}
whence $T_*\leq C\e^2$. Thus, we obtain
\begin{equation}\label{t_max_higher_dim}
t_*=1+T_*+\frac{n-4}{2(n-2)}T^2_*+o(\e^4).
\end{equation}
Therefore, by \eqref{eq:A}, \eqref{eq:E}, \eqref{t_max_higher_dim}  and \eqref{max_I_at_tu}, we conclude that
\begin{align}
&\max_{0\leq t<\infty}I[t\ubar]\no\\
=&\frac{1}{n-1}E[t_*\ubar]+4(n-2)\int_M(t_*\ubar)^{\frac{2n}{n-2}}d\mu_{g_{x_0}}\no\\
=&\frac{4}{n-2}\Big[1+2T_*+\frac{2n-6}{n-2}T_*^2+o(\e^4)\Big](n(n-2)A+cB+\tilde E+o(\e^4))\no\\
&+4(n-2)\Big[1+\frac{2n}{n-2}T_*+\frac{2n(n-1)}{(n-2)^2}T_*^2+o(\e^4)\Big](A+\tilde A+O(\e^5))\no\\
=&S_c+\frac{4}{n-2}\tilde E+\frac{8}{n-2}(2n(n-2)A+cB)T_\ast+4(n-2)\tilde A+\frac{8}{n-2} \tilde ET_\ast\no\\
&+\left[\frac{8(n-3)}{(n-2)^2}(n(n-2)A+cB)+\frac{8n(n-1)}{n-2}A\right]T_*^2+o(\e^4)\no\\
=&S_c+\frac{4(n-1)}{n-2}\tilde E-4c\tilde B-4(n-1)(n-2)\tilde A\no\\
&+\frac{c^2\tilde B^2}{n(n-2)A}-\frac{(c^2B\tilde B+2n(n-2)A\tilde E)^2}{n(n-2)A(2n(n-2)A+cB)^2}+\frac{4\tilde E(\tilde E-c\tilde B)}{2n(n-2)A+cB}\no\\
&+\left[\frac{8(n-3)}{(n-2)^2}(n(n-2)A+cB)+\frac{8n(n-1)}{n-2}A\right]T_*^2+o(\e^4)\no\\
=&S_c+\frac{4(n-1)}{n-2}\tilde E-4c\tilde B-4(n-1)(n-2)\tilde A\no\\
&+\frac{4(n(n-2)A+cB)}{(2n(n-2)A+cB)^2}(\tilde E-c\tilde B)^2+o(\e^4)\no\\
&+\left[2(n-3)(n(n-2)A+cB)+2n(n-1)(n-2)A\right]\frac{(\tilde E-c\tilde B)^2}{(2n(n-2)A+cB)^2}\no\\
=&S_c+\frac{4(n-1)}{n-2}\tilde E-4c\tilde B-4(n-1)(n-2)\tilde A+\frac{2(n-1)(\tilde E-c\tilde B)^2}{2n(n-2)A+cB}+o(\e^4).\label{eq:level_n=6,7}
\end{align}

We are now in a position to verify \eqref{eq:A}-\eqref{eq:E}.

\begin{lemma}\label{lem:volume_in_n=6,7} Assume that $n\geq 7$ and $\pa M$ is umbilic, then if $0<\e\ll\rho<\rho_0$ for some sufficiently small $\rho_0$, the volume of $M$ has the expansion
\begin{align}\label{eq:volume_in_n=6,7}
\int_M \ubar^{\frac{2n}{n-2}}d\mu_{g_{x_0}}=&A+A_2+O(\rho^{-6}\e^6),
\end{align}
where
$$A_2=\frac{n(n+2)}{(n-2)^2}\int_{B_\rho^+}\U^{\frac{4}{n-2}}\w^2dy$$
with $|A_2|\leq C\e^4 $.
\end{lemma}
\begin{proof}
The proof follows from the same lines in Lemma \ref{lem:exp_vol_M} but with the estimate \eqref{est:psi_higher_dim}.
\end{proof}
\begin{lemma}\label{lem:volume_bd_n=6,7}
Assume that $n\geq 7$ and $\pa M$ is umbilic, then if $0<\e\ll\rho<\rho_0$ for some sufficiently small $\rho_0$, the volume on the boundary has 
\begin{align}\label{eq:volume_bd_n=6,7}
\int_{\pa M}\ubar^{\frac{2(n-1)}{n-2}}d\sigma_{g_{x_0}}=&B+B_1+B_2+O(\rho^{-6}\e^6),
\end{align}
where
\begin{align*}
B_1=&\frac{2(n-1)}{n-2}\int_{D_\rho}\U^{\frac{n}{n-2}}\psi d\sigma=\frac{1}{2}\int_{D_\rho}\U^{\frac{2(n-1)}{n-2}}S_{nn}d\sigma+O(\e^{n-1}\rho^{3-n}),\\
B_2=&\frac{n(n-1)}{(n-2)^2}\int_{D_\rho}\U^{\frac{2}{n-2}}\w^2d\sigma
\end{align*}
with $|B_1|\leq C(n,T_c)\e^2, |B_2|\leq C(n,T_c)\e^4$.
\end{lemma}
\begin{proof}
By \eqref{vol_conf_Fermi_coor} we have 
\begin{align*}
\int_{\partial M}\ubar^{\frac{2(n-1)}{n-2}}d\sigma_{g_{x_0}}=&\int_{\partial\Omega_\rho\cap\partial M}((\U+\psi)\circ\Psi_{x_0}^{-1})^{\frac{2(n-1)}{n-2}}d\sigma_{g_{x_0}}+O(\e^{n-1}\rho^{1-n})\\
=&\int_{D_\rho }(\U+\psi)^{\frac{2(n-1)}{n-2}}d\sigma+O(\e^{2d+2})+O(\e^{n-1}\rho^{1-n}).
\end{align*}
By choosing $\rho$ small enough and \eqref{est:psi_higher_dim} we have 
\begin{align*}
&\int_{D_\rho}(\U+\psi)^{\frac{2(n-1)}{n-2}}d\sigma\\
=&\int_{D_\rho}\U^{\frac{2(n-1)}{n-2}}d\sigma+\frac{2(n-1)}{n-2}\int_{D_\rho}\U^{\frac{n}{n-2}}\psi d\sigma\\
&+\frac{(n-1)n}{(n-2)^2}\int_{D_\rho}\U^{\frac{2}{n-2}}\w^2d\sigma+O(\rho^{-6}\e^6)\\
:=& B_0+B_1+B_2+O(\rho^{-6}\e^6).
\end{align*}
Notice that
\begin{align*}
B_0=B+O(\e^{n-1}\rho^{1-n}).
\end{align*}
Then it follows from \eqref{eq:def:psi} and \eqref{eq:V} that
\begin{align*}
\frac{2(n-1)}{n-2}\U^{\frac{n}{n-2}}\w=\sum_{a=1}^{n-1}\partial_a(\U^{\frac{2(n-1)}{n-2}}V_a)+\frac{1}{2}\U^{\frac{2(n-1)}{n-2}}S_{nn}
\end{align*}
on $D_\rho$, an integration by parts gives
\begin{align*}
B_1=&\frac{1}{2}\int_{D_\rho}\U^{\frac{2(n-1)}{n-2}}S_{nn}d\sigma+\rho^{-1}\int_{\pa D_\rho}\U^{\frac{2(n-1)}{n-2}}V_a y^a d\sigma\no\\
=&\frac{1}{2}\int_{D_\rho}\U^{\frac{2(n-1)}{n-2}}S_{nn}d\sigma+O(\rho^{3-n}\e^{n-1})
\end{align*}
and $|B_1|=O(\e^2)$ by virtue of \eqref{est:V} and \eqref{est:h_higher_dim}.
\end{proof}

 However, it is a little bit tricky to estimate $E[\ubar]$ in this case. By adopting the same notation in \cite[formula (5.33)]{Chen-Sun}, we decompose
\begin{align}\label{eq:decom_energy_integrand}
&\frac{4(n-1)}{n-2}|\nabla(\U+\w)|^2_{g_{x_0}}+R_{g_{x_0}}(\U+\w)^2\no\\
=&\frac{4(n-1)}{n-2}|\nabla \U|^2+\frac{4(n-1)}{n-2}n(n+2)\U^{\frac{4}{n-2}}\w^2+\sum_{i=1}^4 J_i,
\end{align}
where $J_i, 1 \leq i \leq 4$ were defined in \cite[P.33]{Chen-Sun}.
By \cite[estimate (5.34)]{Chen-Sun}  and Lemma \ref{lem:volume_bd_n=6,7}, we obtain 
\begin{align}
\int_{B_\rho^+}J_1dy\leq&-\frac{8(n-1)}{n-2}\int_{D_\rho}\pa_n\U \w d\sigma+\int_{\pa^+B_\rho^+}(\U^2\pa_k h_{ik}-\pa_k\U^2 h_{ik})\frac{y^i}{|y|}d\sigma\nonumber\\
&+C\sum_{a,b=1}^{n-1}\sum_{|\alpha|=2}^d|\pa^\alpha h_{ab}|\epsilon^{n-2}\rho^{|\alpha|+2-n}+C\rho^{2d+4-n}\epsilon^{n-2}\no\\
=&\frac{8(n-1)c}{n-2}\int_{D_\rho}\U^{\frac{n}{n-2}}\w d\sigma+\int_{\pa^+B_\rho^+}(\U^2\pa_k h_{ik}-\pa_k\U^2 h_{ik})\frac{y^i}{|y|}d\sigma\no\\
&+O(\e^{n-2}\rho^{4-n})\no\\
=&4cB_1+\int_{\pa^+B_\rho^+}(\U^2\pa_k h_{ik}-\pa_k\U^2 h_{ik})\frac{y^i}{|y|}d\sigma+O(\e^{n-2}\rho^{4-n}),\label{est:J_1}
\end{align}
where the first identity follows from \eqref{prob:half-space}: $\partial_n\U=(n-2)T_c\U^{\frac{n}{n-2}}=-c\U^{\frac{n}{n-2}}$ on $D_\rho$. Using an intermediate estimate in \cite[estimate (5.35)]{Chen-Sun}, we obtain
\begin{align}\label{est:J_2}
&\int_{B_\rho^+}J_2dy\no\\
=&\overset{-K}{\overbrace{-\frac{1}{4}\int_{B_\rho^+}Q_{ik,l}Q_{ik,l}dy-2\int_{B_\rho^+}\U^{\frac{2n}{n-2}}T_{ik}T_{ik} dy}}+\int_{\pa^+B_\rho^+}\xi_i\frac{y^i}{|y|}d\sigma-\int_{D_{\rho}}\xi_nd\sigma\no\\
=&-K+O(\e^{n-2}\rho^{6-n})+\frac{n+2}{2(n-2)}\int_{D_\rho}\U\pa_n\U S_{nn}^2d\sigma+4cB_2,
\end{align}
where the vector field $\xi$ was defined in \cite[Proposition 5.3]{Chen-Sun} and we have used
$$\int_{\pa^+B_\rho^+}\xi_i\frac{y^i}{|y|}d\sigma=O(\e^{n-2}\rho^{6-n})$$
and the exact expression of $\xi_n$ (see  \cite[(5.28)]{Chen-Sun}). Since $\pa M$ is umbilic, using an independent estimate in \cite[Corollary 12]{Brendle2}, we obtain an improved estimate of $J_3$ and $J_4$ (c.f. \cite[estimate (5.37)]{Chen-Sun})
\begin{align*}
&J_3+J_4\\
\leq&C \sum_{a,b=1}^{n-1}\sum_{|\alpha|=2}^d (|\pa^\alpha h_{ab}|^2(\e+|y|)^{2|\alpha|+4-2n}+|\pa^\alpha h_{ab}|(\e+|y|)^{|\alpha|+d+3-2n})\e^{n-2}\\
&+C(\e+|y|)^{2d+4-2n}\e^{n-2}.
\end{align*}
Thus, we have
\begin{align}\label{est:J_3+J_4}
\int_{B_\rho^+} (J_3+J_4)dy\leq 
\begin{cases}
C\e^{n-2}\log \frac{\rho}{\e}, &\mathrm{~~if~~} n=7,8,\\
C \e^6, &\mathrm{~~if~~} n\geq 9.
\end{cases}
\end{align}
Moreover, by \eqref{vol_conf_Fermi_coor} and \eqref{est:h_higher_dim} we have
\begin{equation}\label{est:mean_curv_hot}
\int_{D_\rho}h_{g_{x_0}}(\U+\w)^2 d\sigma\leq C \int_{D_\rho}|y|^{2d+1}(\U+\w)^2 d\sigma\leq C \epsilon^{n-2}\rho^{2d+4-n}.
\end{equation}
Putting the above estimates \eqref{est:J_1}-\eqref{est:mean_curv_hot} together, from \eqref{eq:decom_energy_integrand} we obtain
\begin{align}\label{est:E_inside_higher_dim}
&E[\ubar;\Omega_\rho]\no\\
=&\frac{4(n-1)}{n-2}\int_{B_\rho^+}|\nabla \U|^2dy+4(n-1)(n-2)A_2+4c(B_1+B_2)+O(\e^5\log \frac{\rho}{\e})\no\\
&-K+\frac{n+2}{2(n-2)}\int_{D_\rho}\U\pa_n\U S_{nn}^2d\sigma+\int_{\pa^+B_\rho^+}(\U^2\pa_k h_{ik}-\pa_k\U^2 h_{ik})\frac{y^i}{|y|}d\sigma.
\end{align}
Using \cite[estimate (5.48)]{Chen-Sun}, when $\e\ll\rho<\rho_0$ for some sufficiently small $\rho_0$ we have
\begin{align}\label{est:E_outside_higher_dim}
&E[\ubar;M\backslash\Omega_\rho]\no\\
\leq &\frac{4(n-1)}{n-2}\int_{\partial^+B_\rho^+}\left[-\partial_i\U\U+\partial_j\U\U h_{ij}-\epsilon^{\frac{n-2}{2}}(\U\partial_iG-G\partial_i\U)\right]\frac{y^i}{|y|} d\sigma\no\\
&+O(\e^{n-2}\rho^{4-n}).
\end{align}
Testing problem \eqref{prob:half-space} with $\U$ and integrating over $B_\rho^+$, we obtain
 \begin{align}\label{est:int_bubble_B_rho^+}
&\int_{B_\rho^+}|\nabla \U|^2dy-\int_{\pa^+B_\rho^+}\U\pa_i \U \frac{y^i}{|y|} d\sigma\no\\
=&n(n-2)\int_{B_\rho^+}\U^{\frac{2n}{n-2}}dy-(n-2)T_c\int_{D_\rho}\U^{\frac{2(n-1)}{n-2}}d\sigma\no\\
=&n(n-2)A+cB+O(\e^{n-1}\rho^{1-n}).
\end{align}
Combining estimates \eqref{est:E_inside_higher_dim} and \eqref{est:E_outside_higher_dim}, when $\e<\rho\ll \rho_0$ for some sufficiently small $\rho_0$, we conclude that
\begin{align}
&E[\ubar]\no\\
=&\frac{4(n-1)}{n-2}\left(\int_{B_\rho^+}|\nabla \U|^2dy-\int_{\partial^+B_\rho^+}\partial_i\U\U\frac{y^i}{|y|} d\sigma\right)\no\\
&+4(n-1)(n-2)A_2+4c(B_1+B_2)-K\no\\
&+\frac{n+2}{2(n-2)}\int_{D_\rho}\U\pa_n\U S_{nn}^2d\sigma+\underline{\int_{\pa^+B_\rho^+}(\U^2\pa_j h_{ij}+\frac{n}{n-2}\pa_j\U^2 h_{ij})\frac{y^i}{|y|}d\sigma}\no\\
&\underline{-\frac{4(n-1)}{n-2}\epsilon^{\frac{n-2}{2}}\int_{\partial^+ B_\rho^+}(\U\partial_iG-G\partial_i\U)\frac{y^i}{|y|}d\sigma}+O(\e^5\log \frac{\rho}{\e})\no\\
=&\frac{4(n-1)}{n-2}[n(n-2)A+cB]+4(n-1)(n-2)A_2+4c(B_1+B_2)\no\\&-K+\frac{n+2}{2(n-2)}\int_{D_\rho}\U\pa_n\U S_{nn}^2d\sigma+O(\e^5\log \frac{\rho}{\e}). \label{eq:energy_n=6,7}
\end{align}
Here the second identity follows from \eqref{est:int_bubble_B_rho^+} and a rough estimate of \cite[(5.50)]{Chen-Sun}, which indicates that the underlined terms are of order $O(\e^{n-2})$.

Comparing \eqref{eq:volume_in_n=6,7}, \eqref{eq:volume_bd_n=6,7}, \eqref{eq:energy_n=6,7}  and \eqref{eq:A}, \eqref{eq:B}, \eqref{eq:E} respectively, we write
\begin{align*}
\tilde A=&A_2, \quad \tilde B=B_1+B_2,\\
\frac{4(n-1)}{n-2}\tilde E=&4(n-1)(n-2)A_2+4c\tilde B-K+\frac{n+2}{2(n-2)}\int_{D_\rho}\U\pa_n\U S_{nn}^2 d\sigma.
\end{align*}
Hence we have $|\tilde A|\leq C\e^4,|\tilde B|\leq C\e^2$ and $\tilde E\leq C\e^2$ as required. On the other hand, by \eqref{est:V} and \eqref{def:S:T}, we estimate
$$K=O(\e^4)\quad\mathrm{and}\quad \int_{D_\rho}\U\pa_n\U S_{nn}^2 d\sigma=O(\e^4).$$
Consequently, keeping in mind that $A_2=O(\e^4)$ by Lemma \ref{lem:volume_in_n=6,7} and $B_1=O(\e^2), B_2=O(\e^4)$ by Lemma \ref{lem:volume_bd_n=6,7}, we plug these estimates into \eqref{eq:level_n=6,7} to get
\begin{align*}
&\max_{0\leq t<\infty}I[t\ubar]\\
=&S_c-K+\frac{n+2}{2(n-2)}\int_{D_\rho}\U\pa_n\U S_{nn}^2 d\sigma+\frac{2(n-1)(\tilde E-c\tilde B)^2}{2n(n-2)A+cB}+o(\e^4)\\
\leq &S_c-K+\frac{n+2}{2(n-2)}\int_{D_\rho}\U\pa_n\U S_{nn}^2 d\sigma+\Lambda c^2B_1^2+o(\e^4),
\end{align*}
where $\Lambda$ can be chosen as
\begin{align}\label{eq:L-precise}
\Lambda=\frac{2}{(n-1)(2n(n-2)A+cB)}.
\end{align}

In order to prove \eqref{ieq:goal_less}, it remains to show
\begin{align*}
\frac{n+2}{2(n-2)}\int_{D_\rho}\U\pa_n\U S_{nn}^2 d\sigma+\Lambda c^2B_1^2< K+o(\e^4)
\end{align*}
for sufficiently small $\e>0$. Recall definitions of $B_1$ and $K$, it suffices to show
\begin{align*}
&\frac{n+2}{2(n-2)}\int_{D_\rho}\U\pa_n\U S_{nn}^2 d\sigma+\frac{\Lambda}{4}\left(\int_{D_\rho}\U\pa_n\U S_{nn}d\sigma\right)^2\\
< &\frac{1}{4}\int_{B_\rho^+}Q_{ik,l}Q_{ik,l}dy+2\int_{B_\rho^+}\U^{\frac{2n}{n-2}}T_{ik}T_{ik} dy+o(\e^4).
\end{align*}

Notice that $S_{nn}+T_{nn}=H_{nn}=0$ on $D_\rho$, then the above inequality becomes
\begin{align}\label{est:goal_higher_dim}
&\frac{n+2}{2(n-2)}\int_{D_\rho}\U\pa_n\U T_{nn}^2 d\sigma+\frac{\Lambda}{4}\left(\int_{D_\rho}\U\pa_n\U T_{nn}d\sigma\right)^2\no\\
< &\frac{1}{4}\int_{B_\rho^+}Q_{ik,l}Q_{ik,l}dy+2\int_{B_\rho^+}\U^{\frac{2n}{n-2}}T_{ik}T_{ik}dy+o(\e^4)
\end{align}
for any $c \in \mathbb{R}$, when $0<\e\ll\rho<\rho_0$ for some sufficiently small $\rho_0$.

\begin{theorem}
Assume that $n\geq 7$, $\pa M$ is umbilic and the Weyl tensor $W_{g_0}$ of $M$ is non-zero at some $x_0\in \pa M$. Then problem \eqref{eq:main_eq_on_M} is solvable for all non-negative constant $c$.
\end{theorem}
\begin{proof}
Based on the above estimates, it reduces to showing \eqref{est:goal_higher_dim} for all non-negative constant $c$.
Without loss of generality we assume that $T_{nn}\not \equiv 0$ on $D_\rho$, otherwise it is trivial. By H\"older's inequality we have
\begin{align*}
\left(\int_{D_\rho}\U\pa_n\U T_{nn}d\sigma\right)^2\leq \int_{D_\rho}\U(-\pa_n\U) d\sigma\int_{D_\rho}\U(-\pa_n\U) T_{nn}^2d\sigma,
\end{align*}
since from \eqref{prob:half-space} that $\pa_n\U=-c\U^{\frac{n}{n-2}}$ on $D_\rho$. Together with \eqref{eq:L-precise} and $n\geq 7$, we obtain
\begin{align*}
\frac{\Lambda}{4}\int_{D_\rho}\U(-\partial_n\U)d\sigma\leq \frac{\Lambda}{4}cB\leq \frac{1}{2(n-1)}< \frac{n+2}{4(n-2)}.
\end{align*}
Then for any $c\geq 0$,
\begin{align}
\mathrm{LHS~~of~~}\eqref{est:goal_higher_dim}\leq \frac{n+2}{4(n-2)}\int_{D_\rho}\U\pa_n\U T_{nn}^2 d\sigma\leq 0.\label{est:LHS_positive_c}
\end{align}

On the other hand, thanks to estimate \eqref{est:bubble_fcn}, \eqref{est:h_higher_dim} and the assumption that $\pa M$ is umbilic, we can apply Proposition \ref{app:prop:1} to show
\begin{align}
\int_{B_\rho^+}Q_{ik,l}Q_{ik,l}dy\geq \lambda^\ast \e^{n-2}\int_{B_\rho^+}(\e+|y|)^{6-2n}dy\geq C\lambda^\ast \e^4 \label{eq:Q_S_e4}
\end{align}
for all $\rho \geq 2(1+|T_c|)\e$, where $\lambda^\ast=\lambda^\ast(n)>0$ and we have used \eqref{est:h_higher_dim} in the first inequality. Then \eqref{eq:Q_S_e4} implies that
\begin{align}
\frac{1}{4}\int_{B_\rho^+}Q_{ik,l}Q_{ik,l}dy+2\int_{B_\rho^+}\U^{\frac{2n}{n-2}}T_{ik}T_{ik} dy\geq C_1\e^4, \label{est:RHS_positive_c}
\end{align}
where $C_1=C_1(n)>0$. Hence the estimate \eqref{est:goal_higher_dim} follows from \eqref{est:LHS_positive_c} and \eqref{est:RHS_positive_c}.
\end{proof}

For any non-negative constant $c$, the selection of $\Lambda$ and the sign of $c$ are very crucial in the above verification of estimate \eqref{est:goal_higher_dim}.
However, at present we are not sure whether \eqref{est:goal_higher_dim} is true for all negative constants $c$. Instead, we consider the inequality \eqref{est:goal_higher_dim} on a spherical cap. This is realized by a pull-back of a stereographic projection from the spherical cap combined with an application of the sharp Sobolev trace inequality.

\begin{theorem}\label{thm:negative_c_higher_dim}
Let $n\geq 7$ and $c$ be a negative real number. Assume that $\pa M$ is umbilic and the Weyl tensor $W_{g_0}$ of $M$ is non-zero at some $x_0\in \pa M$, then problem \eqref{eq:main_eq_on_M} is solvable for all $c \in [-c_0,0)$, where $c_0$ is a positive dimensional constant.
\end{theorem}
\begin{proof}
It suffices to prove the above inequality \eqref{est:goal_higher_dim} when $B_\rho^+$ and $D_\rho$ are replaced by $\mathbb{R}^n_+$ and $\mathbb{R}^{n-1}$, respectively, because the integrals outside these domains are $o(\e^4)$ through direct computations.

Let $\pi: S^n(T_c\textbf{e}_n) \setminus\{T_c \textbf{e}_n+\textbf{e}_{n+1}\} \to \{y+T_c \textbf{e}_n \in \mathbb{R}^{n+1};y^{n+1}=0\}\simeq \mathbb{R}^n$ be the stereographic projection from the unit sphere $S^n(T_c\textbf{e}_n)$ in $\mathbb{R}^{n+1}$ centered at $T_c \textbf{e}_n$. Let $\Sigma:=\pi^{-1}(\mathbb{R}^n_+) \subset S^n$ denote a spherical cap equipped with the metric $g_{\Sigma}=\frac{1}{4}g_{S^n}$, where $g_{S^n}$ is the standard round metric of $S^n$. Denote by $\nu=\nu_{g_\Sigma}$ the outward unit normal of $g_{\Sigma}$ on $\pa \Sigma$. We define $\mathcal{T}(x)=(\U^{4/(n-2)}T)\circ \pi(x)$ as a symmetric 2-tensor on $\Sigma$. In particular, we have $\mathcal{T}_{\nu\nu}:=\mathcal{T}(\nu_{g_\Sigma},\nu_{g_\Sigma})=T_{nn}$ on $\pa \Sigma$. Thus, the quantities $A$ and $B$ in \eqref{def:volumes} are the volumes of $(\Sigma, g_{\Sigma})$ and $\pa \Sigma$ with the induced metric of $g_\Sigma$, respectively. We have
\begin{align*}
(\pi^{-1})^\ast(g_\Sigma)=W_\e(y)^{\frac{4}{n-2}}g_{\mathbb{R}^n}
\end{align*}
and define as in Brendle \cite{Brendle2}
\begin{align*}
Q_{ij,k}=&\U \partial_kT_{ij}+\frac{2}{n-2}(\partial_l\U T_{il}\delta_{jk}+\partial_l\U T_{jl}\delta_{ik}-\partial_i\U T_{jk}-\partial_j\U T_{ik})\\
=&\U^{\frac{n-6}{n-2}} \nabla_k^\Sigma (\mathcal T\circ \pi^{-1})_{ij},
\end{align*}
where $\nabla_k^\Sigma$ denotes the covariant derivative with respect to the metric $g_\Sigma$ on $\Sigma$.
Then from \eqref{prob:half-space} we obtain 
\begin{align*}
\int_{\mathbb{R}^n_+}Q_{ik,l}Q_{ik,l}dy=&\int_{\Sigma}|\nabla\mathcal{T}|^2_{g_\Sigma}d\mu_{g_\Sigma},\\
\int_{\mathbb{R}^n_+}\U^{\frac{2n}{n-2}}T_{ik}T_{ik} dy=&\int_{\Sigma}|\mathcal{T}|_{g_\Sigma}^2d\mu_{g_\Sigma},\\
\int_{\mathbb{R}^{n-1}}\U\pa_n\U T_{nn}^2 d\sigma=&-c\int_{\partial\Sigma} \mathcal{T}_{\nu\nu}^2d\sigma_{g_\Sigma}.
\end{align*}
Thus, it suffices to show
\begin{align*}
\int_{\Sigma}(\frac{1}{4}|\nabla \mathcal{T}|_{g_\Sigma}^2+2|\mathcal{T}|_{g_\Sigma}^2)d\mu_{g_\Sigma}+\frac{(n+2)c}{2(n-2)}\int_{\partial \Sigma}|\mathcal{T}|_{g_\Sigma}^2d\sigma_{g_\Sigma}> \frac{\Lambda c^2}{4}\left(\int_{\partial\Sigma} \mathcal{T}_{\nu\nu} d\sigma_{g_\Sigma}\right)^2,
\end{align*}
which implies the estimate \eqref{est:goal_higher_dim}.

The sharp Sobolev trace inequality in \cite[Theorem 1]{escobar7} (see also \cite{li-zhu}) states that
\begin{align}\label{Sobolev_trace_on_half_space}
\frac{4(n-1)}{n-2}\int_{\mathbb{R}_+^n}|\nabla \psi|^2 dy
\geq Q(B^n,S^{n-1},[g_{\mathbb{R}^n}])\left(\int_{\mathbb{R}^{n-1}}|\psi|^{\frac{2(n-1)}{n-2}}d\sigma\right)^{\frac{n-2}{n-1}}
\end{align}
for any $\psi \in C^\infty(\overline{\mathbb{R}_+^n})$. Furthermore, the sharp constant is given by
$$Q(B^n,S^{n-1},[g_{\mathbb{R}^n}])=\frac{n-2}{2}\omega_{n-1}^{\frac{1}{n-1}}.$$
Given any $\varphi \in H^1(\Sigma,g_{\Sigma})$, from the conformal invariance \eqref{eq:conformal_invariance} of both $L_{g_\Sigma}$ and $B_{g_\Sigma}$, we set $\psi(y)=((\varphi\circ \pi^{-1})\U)(y)$ and use \eqref{Sobolev_trace_on_half_space} to show
\begin{align}\label{Sobolev_trace_on_cap}
&\int_{\Sigma}\left(\frac{4(n-1)}{n-2}|\nabla \varphi|_{g_\Sigma}^2+ R_{g_\Sigma} \varphi^2\right)d\mu_{g_\Sigma}+2(n-1)\int_{\partial \Sigma}h_{g_\Sigma}\varphi^2d\sigma_{g_\Sigma}\no\\
\geq& Q(B^n,S^{n-1},[g_{\mathbb{R}^n}])\left(\int_{\pa \Sigma}|\varphi|^{\frac{2(n-1)}{n-2}}d\sigma_{g_\Sigma}\right)^{\frac{n-2}{n-1}}.
\end{align}
Moreover, it is easy to show that $R_{g_\Sigma}=4n(n-1)$ and $h_{g_\Sigma}=2c/(n-2)$ in virtue of \eqref{prob:half-space}.

 By choosing $\varphi=|\mathcal{T}|_{g_\Sigma}$ and using Kato's inequality:
 $$|\nabla \mathcal{T}|_{g_\Sigma}\geq |\nabla|\mathcal{T}|_{g_\Sigma}|_{g_\Sigma},$$
 the above inequality \eqref{Sobolev_trace_on_cap} leads to 
\begin{align*}
&\int_{\Sigma}\left(\frac{4(n-1)}{n-2}|\nabla\mathcal{T}|^2_{g_\Sigma}+4n(n-1)|\mathcal{T}|_{g_\Sigma}^2\right)d\mu_{g_\Sigma}+\frac{4(n-1)c}{n-2}\int_{\partial\Sigma}|\mathcal{T}|_{g_\Sigma}^2d\sigma_{g_\Sigma}\\
\geq&Q(B^n,S^{n-1},[g_{\mathbb{R}^n}])\left(\int_{\pa \Sigma}|\mathcal{T}|_{g_\Sigma}^{\frac{2(n-1)}{n-2}}d\sigma_{g_\Sigma}\right)^{\frac{n-2}{n-1}}.
\end{align*}
Equivalently,
\begin{align*}
&\int_{\Sigma}\left(|\nabla \mathcal{T}|_{g_\Sigma}^2+n(n-2)|\mathcal{T}|_{g_\Sigma}^2\right)d\mu_{g_\Sigma}+c\int_{\partial \Sigma}|\mathcal{T}|^2_{g_\Sigma}d\sigma_{g_\Sigma}\\
\geq&\frac{n-2}{4(n-1)}Q(B^n,S^{n-1},[g_{\mathbb{R}^n}])\left(\int_{\pa \Sigma}|\mathcal{T}|_{g_\Sigma}^{\frac{2(n-1)}{n-2}}d\sigma_{g_\Sigma}\right)^{\frac{n-2}{n-1}}.
\end{align*}
Consequently, we conclude that
\begin{align}\label{est:lhs_higher_dim}
&\int_{\Sigma}(\frac{1}{4}|\nabla \mathcal{T}|_{g_\Sigma}^2+2|\mathcal{T}|_{g_\Sigma}^2)d\mu_{g_\Sigma}+\frac{(n+2)c}{2(n-2)}\int_{\partial \Sigma}|\mathcal{T}|_{g_\Sigma}^2d\sigma_{g_\Sigma}\no\\
\geq&\left(\frac{1}{4}-\frac{2}{n(n-2)}\right)\int_{\Sigma}|\nabla \mathcal{T}|_{g_\Sigma}^2d\mu_{g_\Sigma}+\left(\frac{n+2}{2(n-2)}-\frac{2}{n(n-2)}\right)c\int_{\pa \Sigma}| \mathcal{T}|_{g_\Sigma}^2d\sigma_{g_\Sigma}\no\\
&+\frac{1}{2n(n-1)}Q(B^n,S^{n-1},[g_{\mathbb{R}^n}])\left(\int_{\pa \Sigma}|\mathcal{T}|_{g_\Sigma}^{\frac{2(n-1)}{n-2}}d\sigma_{g_\Sigma}\right)^{\frac{n-2}{n-1}}\no\\
\geq&\left[\frac{n-2}{4n(n-1)}\omega_{n-1}^{\frac{1}{n-1}}+\frac{n^2+2n-4}{2n(n-2)}c B^{\frac{1}{n-1}}\right]\left(\int_{\pa \Sigma}|\mathcal{T}|_{g_\Sigma}^{\frac{2(n-1)}{n-2}}d\sigma_{g_\Sigma}\right)^{\frac{n-2}{n-1}}.
\end{align}

On the other hand, by H\"older's inequality we have
\begin{align}\label{est:rhs_higher_dim}
\frac{\Lambda}{4}\left(\int_{\mathbb{R}^{n-1}}\U\pa_n\U T_{nn}d\sigma\right)^2
\leq &\frac{\Lambda}{4}\int_{\mathbb{R}^{n-1}}\U\pa_n\U d\sigma\int_{\mathbb{R}^{n-1}}\U\pa_n\U T_{nn}^2d\sigma\no\\
\leq& \frac{\Lambda c^2}{4} B^{\frac{n}{n-1}}\left(\int_{\pa \Sigma}|\mathcal{T}|_{g_\Sigma}^{\frac{2(n-1)}{n-2}}d\sigma_{g_\Sigma}\right)^{\frac{n-2}{n-1}}.
\end{align}
In order to show \eqref{est:goal_higher_dim},  our goal is to find which $c<0$ satisfying
\begin{align}\label{ineq:est1_higher_dim}
&\frac{n-2}{4n(n-1)}\omega_{n-1}^{\frac{1}{n-1}}+\frac{n^2+2n-4}{2n(n-2)}c B^{\frac{1}{n-1}}
\geq\frac{1}{n(n-1)(n-2)}\frac{2^n}{\omega_n}c^2 B^{\frac{n}{n-1}}.
\end{align}
Assuming this claim temporarily, we obtain
\begin{align*}
\frac{1}{n(n-1)(n-2)}\frac{2^n}{\omega_n}c^2 B^{\frac{n}{n-1}}>\frac{1}{2(n-1)}\frac{1}{n(n-2)A}c^2 B^{\frac{n}{n-1}}> \frac{\Lambda c^2}{4} B^{\frac{n}{n-1}},
\end{align*}
where the second inequality follows from the fact that $A>\frac{1}{2^{n+1}}\omega_n$ for $c<0$ and the third inequality follows from the definition \eqref{eq:L-precise} of $\Lambda$ and $n(n-2)A+cB>0$. Thus, \eqref{est:goal_higher_dim} follows from \eqref{est:lhs_higher_dim} and \eqref{est:rhs_higher_dim}.

It remains to show \eqref{ineq:est1_higher_dim}. Let us collect some elementary facts on the geometric quantity $B$ for $c<0$. Recall that $T_c=-\frac{c}{n-2}$ and let
$\cos r=\frac{-T_c}{\sqrt{1+T_c^2}}$, then $r \in (\frac{\pi}{2},\pi)$ and $B=B(r)=2^{1-n}\omega_{n-1}(\sin r)^{n-1}$. In terms of the new variable $r$ we obtain
\begin{align}\label{eqs:geometric_qtn}
cB^{\frac{1}{n-1}}=\frac{n-2}{2}\omega_{n-1}^\frac{1}{n-1}\cos r \mathrm{~~and~~} c^2 B^{\frac{n}{n-1}}=\frac{(n-2)^2}{2^n}\omega_{n-1}^\frac{n}{n-1} \cos^2 r \sin^{n-2} r.
\end{align}
Substituting \eqref{eqs:geometric_qtn} into \eqref{ineq:est1_higher_dim}, we conclude that the inequality \eqref{est:goal_higher_dim} holds for all $r \in (\frac{\pi}{2},\pi)$ satisfying
$$1+\frac{(n-1)(n^2+2n-4)}{n-2}\cos r\geq 4 \frac{\omega_{n-1}}{\omega_n}\cos^2 r \sin^{n-2} r.$$
Therefore, we conclude from the above inequality that the inequality \eqref{est:goal_higher_dim} holds for all $c \in [-c_0, 0)$, where $c_0=c_0(n)>0$ is also determined by the above inequality.
 \end{proof}

\begin{remark}
Though we tend to believe that the test function \eqref{test_fcn_Chen_Sun} should be one of good choices of \eqref{ieq:goal_less}, all our attempts in this direction up to now have failed. As mentioned earlier, we will use a different type of test function in Subsection \ref{Subsect5.2} to verify \eqref{ieq:goal_less} in the remaining case of dimension $n=7$.
\end{remark}

\subsection{Non-positive constant boundary mean curvature}\label{Subsect5.2}

We first present some preliminary results.
\begin{proposition}\label{prop:expan_Fermi_metrics}
Suppose that $n \geq 7$ and $\pa M$ is umbilic, then under $g_{x_0}$-Fermi coordinates around $x_0 \in \pa M$, there hold $g_{x_0}^{an}=0$, $g_{x_0}^{nn}=1$ and
\begin{align*}
g_{x_0}^{ab}=&\delta_{ab}+\frac{1}{3}\bar{R}_{acbd}y^cy^d+R_{nanb} (y^n)^2\\
&+\frac{1}{6}\bar{R}_{acbd,e}y^cy^d y^e+R_{nanb,c}(y^n)^2y^c+\frac{1}{3}R_{nanb,n}(y^n)^3\\
&+\left(\frac{1}{20}\bar{R}_{acbd,ef}+\frac{1}{15}\bar{R}_{achd}\bar{R}_{behf}\right)y^cy^d y^e y^f\\
&+\left(\frac{1}{2}R_{nanb,cd}+{\rm Sym_{ab}}(\bar R_{aced} R_{nenb})\right)(y^n)^2 y^cy^d+\frac{1}{3}R_{nanb,nc}(y^n)^3 y^c\\
&+\frac{1}{12}(R_{nanb,nn}+8R_{nanc}R_{ncnb})(y^n)^4+O(|y|^5),
\end{align*}
where the curvature quantities are with respect to metric $g_{x_0}$ and evaluated at $x_0$, $\bar R_{abcd}$ is the Riemannian curvature tensor of $\pa M$. Moreover at $x_0$ there hold
\begin{align*}
&R_{nn}=R_{nn,n}=\bar R_{ab}=0,\quad -R_{,nn}-2(R_{nanb})^2=2R_{nanb,ab},\\
&-\bar \Delta R_{g_{x_0}}=\frac{1}{6}|\overline{W}_{g_{x_0}}|_{g_{x_0}}^2.
\end{align*}
\end{proposition}
\begin{proof}
For $n \geq 7$, then $d \geq 2$. Notice that $h_{g_{x_0}}=O(|y|^{2d+1})$ near $x_0$ and $\pa M$ is umbilic, it yields $|(\pi_{g_{x_0}})_{ab}|=O(|y|^{2d+1})$ and then $\bar \nabla_{\alpha}\pi_{g_{x_0}}(0)=0$ for all $|\alpha| \leq 4$. From this, the first assertion is a direct consequence of Marques \cite[Lemma 2.3]{marques1} (see also \cite[Lemma 2.3]{Almaraz1}). Since $\det{g_{x_0}}=1+O(|x|^{2d+2})$, then the coefficients in the expansion of  $\det{g_{x_0}}$, which can be derived from \cite[Lemma 2.2]{marques1}, vanish up to order $4$. From this, the second assertion follows by the same lines in \cite[Proposition 3.2]{marques1}.
\end{proof}

Let $W$ and $\overline W$ denote the Weyl tensors of $M$ and $\pa M$ of metric $g_{x_0}$ respectively. We restate Almaraz \cite[Lemma 2.5]{Almaraz1}, which is crucial in the following test function construction. 
\begin{lemma}\label{lem:Almaraz}
Suppose that $\pa M$ is umbilic and $n\geq 4$. Then under $g_{x_0}$-Fermi coordinates around $x_0\in \pa M$, $W_{ijkl}(x_0)=0$ if and only if $R_{nanb}(x_0)=0=\overline W_{abcd}(x_0)$.
\end{lemma}

\begin{theorem}\label{thm:umbilic_bdry_n=7_c<0}
Suppose that $n= 7$, $\pa M$ is umbilic and the Weyl tensor of $M$ is non-zero at a boundary point, then problem \eqref{eq:main_eq_on_M} is solvable for all non-positive $c$. 
\end{theorem}
\begin{proof}
We choose a test function as
\begin{equation}\label{test_fcn:non-positive_higher_dim}
\ubar=[\chi_\rho(\U+\phi)]\circ \Psi^{-1}_{x_0},
\end{equation}
where the correction term is
$$\phi(y)=[\kappa_2(y^n-T_c\e)^2+\kappa_1 \e (y^n-\e T_c)+\kappa_0 \e^2)]\e^{\frac{n-2}{2}}R_{nanb}y^ay^b  (\e^2+|y-T_c \e \textbf{e}_n|^2)^{-\frac{n}{2}}$$
and $\kappa_0,\kappa_1,\kappa_2$ are constants only depending on $n,c$ to be determined later. Note that for $c \in \mathbb{R}$, $(\U+\phi)(y)=(1+O(|y|^2+\e^2))\U(y)$ in $B_{2\rho}^+$. Then for sufficiently small $\rho>0$, there holds
$$\frac{1}{2}\U(y)\leq \ubar(\Psi_{x_0}(y)) \leq 2 \U(y)\quad \mathrm{in~~} B_{2\rho}^+.$$
We remark that our test function includes the one used by Almaraz \cite{Almaraz1}, both somewhat inspired by the idea of Marques \cite{marques1}.

 In the following, we will use an elementary identity (see e.g., \cite[P. 390]{marques3}): Let $q$ be a homogeneous polynomial of degree $k$ and $r>0$, then
\begin{equation}\label{est:homo_polynomial}
\int_{S^{n-2}_r} q d\sigma=\frac{r^2}{k(n+k-3)}\int_{S_r^{n-2}}\Delta q d\sigma.
\end{equation}
Using $\Delta^2(R_{nanb}R_{ncnd}y^ay^by^cy^d)=16(R_{nanb})^2$ and \eqref{est:homo_polynomial}, we have
\begin{equation}\label{est:symm_Rm_on_ball}
\int_{S^{n-2}_r}R_{nanb}R_{ncnd}y^ay^by^cy^d d\sigma=\frac{2\omega_{n-2}(R_{nanb})^2}{(n-1)(n+1)}r^{n+2},
\end{equation}
where $\omega_{n-2}$ is the volume of the standard unit sphere $S^{n-2}$.
Recall that
$$\int_0^\infty\frac{x^{\alpha-1}}{(1+x)^{\alpha+\beta}}dx=B(\alpha,\beta)=\frac{\Gamma(\alpha)\Gamma(\beta)}{\Gamma(\alpha+\beta)}$$
for ${\rm Re}(\alpha)>0, {\rm Re}(\beta)>0$. 

A similar argument in Section \ref{Sect4} gives
$$E[\ubar;M\backslash \Omega_\rho]\leq C\e^{n-2}\rho^{2-n}.$$

We turn to estimate $E[\ubar;\Omega_\rho]$. Notice that $R_{nn}=0$ at $x_0$ by Proposition \ref{prop:expan_Fermi_metrics} and then
\begin{equation}\label{est:elementary_high_dim}
\int_{S_r^{n-2}}R_{nanb}y^ay^b d\sigma=\frac{R_{nn}}{n-1}\omega_{n-2}r^{n}=0.
\end{equation}
Since $n= 7$, we obtain
\begin{align*}
&\int_{\Omega_\rho}\ubar^{\frac{2n}{n-2}}d\mu_{g_{x_0}}\\
=&\int_{B_\rho^+}(\U+\phi)^{\frac{2n}{n-2}}dy+O(\e^n\rho^{N-n})\\
=&\int_{B_\rho^+}\U^{\frac{2n}{n-2}}dy+\frac{2n}{n-2}\int_{B_\rho^+}\U^{\frac{n+2}{n-2}}\phi dy+\frac{n(n+2)}{(n-2)^2}\int_{B_\rho^+}\U^{\frac{4}{n-2}}\phi^2dy+O(\e^6)\\
=&\int_{B_\rho^+}\U^{\frac{2n}{n-2}}dy+J+O(\e^6),
\end{align*}
where 
\begin{align*}
J=&\frac{n(n+2)}{(n-2)^2}\int_{\mathbb{R}^n_+}\U^\frac{4}{n-2}\phi^2 dy\\
=&\frac{n(n+2)}{(n-2)^2}\e^4 R_{nanb}R_{ncnd}\int_{\mathbb{R}^n_+}\frac{y^ay^by^cy^d[\kappa_2(y^n-T_c)^2+\kappa_1  (y^n-T_c)+\kappa_0 ]^2}{(1+|y-T_c\textbf{e}_n|^2)^{n+2}} dy\\
=&\frac{2n(n+2)}{(n-2)^2(n^2-1)}(R_{nanb})^2 \e^4\int_{\mathbb{R}^n_+}\frac{|\bar y|^4[\kappa_2(y^n-T_c)^2+\kappa_1 (y^n-T_c)+\kappa_0]^2}{(1+|y-T_c\textbf{e}_n|^2)^{n+2}}dy\\
=&\frac{n(n+2)\omega_{n-2}\e^4}{(n-2)^2(n^2-1)}B(\tfrac{n+3}{2},\tfrac{n+1}{2})(R_{nanb})^2\\
&\cdot \int_0^\infty \frac{[\kappa_2(y^n-T_c)^2+\kappa_1  (y^n-T_c)+\kappa_0]^2}{(1+(y^n-T_c)^2)^{\frac{n+1}{2}}}dy^n
\end{align*}
by using \eqref{est:symm_Rm_on_ball} and
\begin{align*}
\frac{2n}{n-2}\int_{B_\rho^+}\U^{\frac{n+2}{n-2}}\phi dy=0
\end{align*}
by virtue of \eqref{est:elementary_high_dim}.
 Similarly we obtain
\begin{align*}
&\int_{\Omega_\rho \cap \partial M}\bar{U}_{(x_0,\e)}^{\frac{2(n-1)}{n-2}}d\sigma_{g_{x_0}}\\
=&\int_{D_\rho}(\U+\phi)^{\frac{2(n-1)}{n-2}}d\sigma+O(\e^{n-1}\rho^{N-n+1})\\
=&\int_{D_\rho}\U^{\frac{2(n-1)}{n-2}}d\sigma+\frac{2(n-1)}{n-2}\int_{D_\rho}\U^{\frac{n}{n-2}}\phi d\sigma\\
&+\frac{n(n-1)}{(n-2)^2}\int_{D_\rho}\U^{\frac{2}{n-2}}\phi^2 d\sigma+O(\e^{6}\log \frac{\rho}{\e})\\
=&\int_{D_\rho}\U^{\frac{2(n-1)}{n-2}}d\sigma+\hat B+O(\e^{6}\log \frac{\rho}{\e}),
\end{align*}
where
\begin{align*}
\hat B&=\frac{n(n-1)}{(n-2)^2}\int_{\R^{n-1}}\U^{\frac{2}{n-2}}\phi^2 d\sigma\\
&=\frac{n(n-1)}{(n-2)^2}\e^4R_{nanb}R_{ncnd}\int_{\R^{n-1}}\frac{y^a y^b y^c y^d(\kappa_2T_c^2-\kappa_1 T_c+\kappa_0)^2}{(1+T_c^2+\bar{y}^2)^{n+1}} d\sigma\\
&=\frac{2n}{(n-2)^2(n+1)}\e^4(R_{nanb})^2\int_{\R^{n-1}}\frac{|\bar{y}|^4(\kappa_2T_c^2-\kappa_1 T_c+\kappa_0)^2}{(1+T_c^2+\bar{y}^2)^{n+1}} d\sigma\\
&=\frac{n\omega_{n-2}}{(n-2)^2(n+1)}B\lp\tfrac{n+3}{2},\tfrac{n-1}{2}\rp\e^4(R_{nanb})^2\frac{(\kappa_2T_c^2-\kappa_1 T_c+\kappa_0)^2}{(1+T_c^2)^{\frac{n-1}{2}}}
\end{align*}
and
$$\frac{2(n-1)}{n-2}\int_{D_\rho}\U^{\frac{n}{n-2}}\phi d\sigma=0$$
by virtue of \eqref{est:elementary_high_dim}.
Observe that
\begin{align*}
&\int_{\Omega_\rho}|\nabla \ubar|_{g_{x_0}}^2d\mu_{g_{x_0}}=\int_{B_\rho^+}|\nabla(\U+\phi)|_{g_{x_0}}^2dy+O(\e^{n-2}\rho^{N+2-n}) \\
=&\int_{B_\rho^+}|\nabla (\U+\phi)|^2dy+\int_{B^+_\rho}(g_{x_0}^{ab}-\delta_{ab})\partial_a(\U+\phi)\partial_b(\U+\phi) dy\\
&+O(\e^{n-2}\rho^{N+2-n}).
\end{align*}
By \eqref{prob:half-space}, an integration by parts gives
\begin{align*}
&\frac{8(n-1)}{n-2}\int_{B_\rho^+}\langle\nabla \U,\nabla \phi\rangle dy \\
=&-\frac{8(n-1)}{n-2}\int_{B_\rho^+}\Delta \U \phi dy-\frac{8(n-1)}{n-2}\int_{D_\rho}\phi \frac{\pa \U}{\pa y^n} d\sigma\\
&+\frac{8(n-1)}{n-2}\int_{\partial^+ B_\rho^+}\phi \frac{\partial \U}{\partial r} d\sigma \\
=&O(\e^{n-2}\rho^{4-n}).
\end{align*}
From this we obtain
\begin{align*}
&\int_{B_\rho^+}|\nabla (\U+\phi)|^2dy\\
=&\int_{B_\rho^+}|\nabla \U|^2 dy+\int_{B_\rho^+}|\nabla \phi|^2 dy+O(\e^{n-2}\rho^{4-n}).
\end{align*}
Direct computations give
\begin{align*}
&|\nabla \phi|^2\\
=&\Big[4R_{nena}R_{nenb}y^a y^bD^2(\e^2+|y-T_c\e \textbf{e}_n|^2)^{-n}\\
&~~+n(n-4)R_{nanb}R_{ncnd}y^a y^b y^c y^dD^2(\e^2+|y-T_c\e \textbf{e}_n|^2)^{-n-1}\\
&~~+R_{nanb}R_{ncnd}y^a y^b y^c y^d(2\kappa_2(y^n-T_c\e)+\kappa_1\e)^2 (\e^2+|y-T_c\e \textbf{e}_n|^2)^{-n}\\
&~~-2n R_{nanb}R_{ncnd}y^a y^b y^c y^d (2\kappa_2(y^n-T_c\e)+\kappa_1\e)\\
&\quad~ \cdot D(y^n-\e T_c)(\e^2+|y-T_c\e \textbf{e}_n|^2)^{-n-1}\\
&~~-n^2\e^2R_{nanb}R_{ncnd}y^a y^b y^c y^dD^2(\e^2+|y-T_c\e \textbf{e}_n|^2)^{-n-2}\Big]\e^{n-2},
\end{align*}
where $D=D(y^n)=\kappa_2(y^n-T_c\e)^2+\kappa_1 \e (y^n-\e T_c)+\kappa_0 \e^2$.

From this, \eqref{est:homo_polynomial} and \eqref{est:symm_Rm_on_ball}, we obtain
\begin{align*}
&\int_{B_\rho^+}|\nabla \phi|^2 dy\\
=&\left\{4R_{nena}R_{nenb}\int_{B_\rho^+}\frac{y^a y^bD(y)^2}{(\e^2+|y-T_c\e \textbf{e}_n|^2)^{n}}dy\right.\\
&\quad+n(n-4)R_{nanb}R_{ncnd}\int_{B_\rho^+}\frac{y^a y^b y^c y^dD(y)^2}{(\e^2+|y-T_c\e \textbf{e}_n|^2)^{n+1}}dy\\
&\quad+R_{nanb}R_{ncnd}\int_{B_\rho^+}\frac{y^a y^b y^c y^d(2\kappa_2(y^n-T_c\e)+\kappa_1\e)^2}{(\e^2+|y-T_c\e \textbf{e}_n|^2)^{n}}dy\\
&\quad-2n R_{nanb}R_{ncnd}\int_{B_\rho^+}\frac{y^a y^b y^c y^d(2\kappa_2(y^n-T_c\e)+\kappa_1\e)D(y)(y^n-\e T_c)}{(\e^2+|y-T_c\e \textbf{e}_n|^2)^{n+1}}dy\\
&\quad\left.-n^2\e^2R_{nanb}R_{ncnd}\int_{B_\rho^+}\frac{y^a y^b y^c y^dD(y)^2}{(\e^2+|y-T_c\e \textbf{e}_n|^2)^{n+2}}dy\right\}\e^{n-2}\\
=&\left\{\frac{4}{n-1}\int_{B_{\rho/\e}^+}\frac{|\bar{y}|^2[\kappa_2(y^n-T_c)^2+\kappa_1(y^n-T_c)+\kappa_0]^2}{(1+|y-T_c\textbf{e}_n|^2)^{n}}dy\right.\\
&\quad+\frac{2n(n-4)}{n^2-1}\int_{B_{\rho/\e}^+}\frac{|\bar{y}|^4[\kappa_2(y^n-T_c)^2+\kappa_1(y^n-T_c)+\kappa_0]^2}{(1+|y-T_c\textbf{e}_n|^2)^{n+1}}dy \\
&\quad+\frac{2}{n^2-1}\int_{B_{\rho/\e}^+}\frac{|\bar{y}|^4[2\kappa_2(y^n-T_c)+\kappa_1]^2}{(1+|y-T_c\textbf{e}_n|^2)^{n}}dy \\
&\quad-\tfrac{4n}{n^2-1}\int_{B_{\rho/\e}^+}\tfrac{|\bar{y}|^4[2\kappa_2(y^n-T_c)+\kappa_1](y^n-T_c)[\kappa_2(y^n-T_c)^2+\kappa_1(y^n-T_c)+\kappa_0]}{(1+|y-T_c\textbf{e}_n|^2)^{n+1}} dy\\
&\quad\left.-\frac{2n^2}{n^2-1}\int_{B_{\rho/\e}^+}\frac{|\bar{y}|^4[\kappa_2(y^n-T_c)^2+\kappa_1(y^n-T_c)+\kappa_0]^2}{(1+|y-T_c\textbf{e}_n|^2)^{n+2}}dy\right\}\e^4(R_{nanb})^2.
\end{align*}
Notice that
\begin{align*}
&\int_{B_{\rho/\e}^+}\frac{|\bar y|^4|y^n|^k}{(1+|y-T_c\textbf{e}_n|^2)^n}dy\\
=&\frac{\omega_{n-2}}{2}B(\tfrac{n+3}{2},\tfrac{n-3}{2})\int_0^\infty \frac{|y^n|^k}{(1+(y^n-T_c)^2)^{\frac{n-3}{2}}}dy^n+O(\rho^{6-n}\e^{n-6}), \quad 0\leq k\leq 2,\\
&\int_{B_{\rho/\e}^+}\frac{|\bar y|^2|y^n|^k}{(1+|y-T_c\textbf{e}_n|^2)^n}dy\\
=&\frac{\omega_{n-2}}{2}B(\tfrac{n+1}{2},\tfrac{n-1}{2})\int_0^\infty \frac{|y^n|^k}{(1+(y^n-T_c)^2)^{\frac{n-1}{2}}}dy^n+O(\rho^{6-n}\e^{n-6}),\quad 0\leq k \leq 4,\\
&\int_{B_{\rho/\e}^+}\frac{|\bar y|^4|y^n|^k}{(1+|y-T_c\textbf{e}_n|^2)^{n+1}}dy\\
=&\frac{\omega_{n-2}}{2}B(\tfrac{n+3}{2},\tfrac{n-1}{2})\int_0^{\infty} \frac{|y^n|^k}{(1+(y^n-T_c)^2)^{\frac{n-1}{2}}}dy^n
+O(\rho^{6-n}\e^{n-6}),\quad 0\leq k \leq 4,\\
&\int_{B_{\rho/\e}^+}\frac{|\bar y|^4|y^n|^k}{(1+|y-T_c\textbf{e}_n|^2)^{n+2}}dy\\
=&\frac{\omega_{n-2}}{2}B(\tfrac{n+3}{2},\tfrac{n+1}{2})\int_0^\infty \frac{|y^n|^k}{(1+(y^n-T_c)^2)^{\frac{n+1}{2}}}dy^n+O(\rho^{4-n}\e^{n-4}), \quad 0\leq k \leq 4.\\
&B(\tfrac{n+1}{2},\tfrac{n-1}{2})=\frac{2n}{n+1}B(\tfrac{n+3}{2},\tfrac{n-1}{2}).
\end{align*}
Therefore, putting these facts together, we conclude that
\begin{align*}
&\int_{B_\rho^+}|\nabla (\U+\phi)|^2dy\\
=&\int_{\mathbb{R}^n_+}|\nabla \U|^2 dy+O(\rho^{4-n}\e^{n-2})\\
&+\left\{n^2B(\tfrac{n+3}{2},\tfrac{n-1}{2})\int_0^\infty \frac{[\kappa_2(y^n-T_c)^2+\kappa_1(y^n-T_c)+\kappa_0]^2}{(1+(y^n-T_c)^2)^{\frac{n-1}{2}}}dy^n\right.\\
&\quad~~+B(\tfrac{n+3}{2},\tfrac{n-3}{2})\int_0^\infty \frac{[2\kappa_2(y^n-T_c)+\kappa_1]^2}{(1+(y^n-T_c)^2)^{\frac{n-3}{2}}} dy^n\\
&\quad~~-2nB(\tfrac{n+3}{2},\tfrac{n-1}{2})\int_0^\infty\tfrac{[2\kappa_2(y^n-T_c)+\kappa_1][\kappa_2(y^n-T_c)^2+\kappa_1(y^n-T_c)+\kappa_0](y^n-T_c)}{(1+(y^n-T_c)^2)^{\frac{n-1}{2}}} dy^n \\
&\quad~~\left.-n^2B(\tfrac{n+3}{2},\tfrac{n+1}{2})\int_0^\infty \tfrac{[\kappa_2(y^n-T_c)^2+\kappa_1(y^n-T_c)+\kappa_0]^2}{(1+(y^n-T_c)^2)^{\frac{n+1}{2}}}dy^n \right\}\tfrac{\omega_{n-2}}{n^2-1}(R_{nanb})^2\e^4.
\end{align*}

Next we need to estimate
\begin{align*}
&\int_{B^+_\rho}(g_{x_0}^{ab}-\delta_{ab})\partial_a(\U+\phi)\partial_b(\U+\phi) dy\\
=&\int_{B^+_\rho}(g_{x_0}^{ab}-\delta_{ab})\partial_a\U\partial_b\U dy+2\int_{B^+_\rho}(g_{x_0}^{ab}-\delta_{ab})\partial_a\U\partial_b\phi dy\\
&+\int_{B^+_\rho}(g_{x_0}^{ab}-\delta_{ab})\partial_a\phi\partial_b\phi dy.
\end{align*}
By the symmetry of the half-ball, \eqref{est:homo_polynomial}, \eqref{est:elementary_high_dim} and Proposition \ref{prop:expan_Fermi_metrics}, we estimate the first term (see also \cite[Lemma A.1]{Almaraz1})
 \begin{align*}
 &\int_{B^+_\rho}(g_{x_0}^{ab}-\delta_{ab})\partial_a\U\partial_b\U dy\\
 =&\frac{(n-2)^2}{n^2-1}\e^4R_{nanb,ab}\int_{B_{\rho/\e}^+}\frac{|y^n|^2|\bar y|^4}{(1+|y-T_c\textbf{e}_n|^2)^n}dy\\
 &+\frac{(n-2)^2}{2(n-1)}\e^4(R_{nanb})^2\int_{B_{\rho/\e}^+}\frac{|y^n|^4|\bar y|^2}{(1+|y-T_c\textbf{e}_n|^2)^n}dy\\
 &+O(\e^5) \int_{B_{\rho/\e}^+}\frac{|y|^7}{(1+|y-T_c\textbf{e}_n|^2)^n}dy\\
 :=&\frac{(n-2)^2}{n^2-1}\e^4R_{nanb,ab}\Theta_1+\frac{(n-2)^2}{2(n-1)}\e^4(R_{nanb})^2\Theta_2+O(\e^5\log \frac{\rho}{\e}),
 \end{align*}
 where 
\begin{align*}
\Theta_1=&\int_{\mathbb{R}^n_+}\frac{|y^n|^2|\bar y|^4}{(1+|y-T_c\textbf{e}_n|^2)^n}dy\\
=&\frac{\omega_{n-2}}{2}B(\tfrac{n+3}{2},\tfrac{n-3}{2})\int_0^\infty \frac{|y^n|^2}{(1+(y^n-T_c)^2)^{\frac{n-3}{2}}}dy^n,\\
\Theta_2=&\int_{\mathbb{R}^n_+}\frac{|y^n|^4|\bar y|^2}{(1+|y-T_c\textbf{e}_n|^2)^n}dy\\
=&\frac{\omega_{n-2}}{2}B(\tfrac{n+1}{2},\tfrac{n-1}{2})\int_0^\infty \frac{(y^n)^4}{(1+(y^n-T_c)^2)^{\frac{n-1}{2}}}dy^n.
\end{align*}
For the second term, by Proposition \ref{prop:expan_Fermi_metrics} direct computations give
 \begin{align*}
&2\int_{B^+_\rho}(g_{x_0}^{ab}-\delta_{ab})\partial_a\U\partial_b\phi dy \\
=&2\int_{B^+_\rho}(g_{x_0}^{ab}-\delta_{ab})\frac{(2-n)y^a\e^{\frac{n-2}{2}}}{(\e^2+|y^n-T_c\e\textbf{e}_n|^2)^{\frac{n}{2}}}\\
&\quad~~\cdot\left\{\frac{2\e^{\frac{n-2}{2}}R_{nenb}y^e[\kappa_2(y^n-T_c\e)^2+\kappa_1(y^n-T_c\e)\e+\kappa_0\e^2]}{(\e^2+|y^n-T_c\e\textbf{e}_n|^2)^{\frac{n}{2}}}\right.\\
&\left.\qquad~~-\frac{nR_{nenf}y^e y^f y^b\e^{\frac{n-2}{2}}[\kappa_2(y^n-T_c\e)^2+\kappa_1(y^n-T_c\e)\e+\kappa_0\e^2]}{(\e^2+|y^n-T_c\e\textbf{e}_n|^2)^{\frac{n}{2}+1}}\right\}dy\\
=&2\int_{\R_+^n}R_{nanb}(y^n)^2\frac{(2-n)y^a\e^{\frac{n-2}{2}}}{(\e^2+|y^n-T_c\e\textbf{e}_n|^2)^{\frac{n}{2}}}\\
&\quad~~\cdot\left\{\frac{2\e^{\frac{n-2}{2}}R_{nenb}y^e[\kappa_2(y^n-T_c\e)^2+\kappa_1(y^n-T_c\e)\e+\kappa_0\e^2]}{(\e^2+|y^n-T_c\e\textbf{e}_n|^2)^{\frac{n}{2}}}\right.\\
&\left.\qquad~-\frac{nR_{nenf}y^e y^f y^b\e^{\frac{n-2}{2}}[\kappa_2(y^n-T_c\e)^2+\kappa_1(y^n-T_c\e)\e+\kappa_0\e^2]}{(\e^2+|y^n-T_c\e\textbf{e}_n|^2)^{\frac{n}{2}+1}}\right\}dy\\
&+O(\e^{5} \log \frac{\rho}{\e})\\
=&-\frac{2n(n-2)}{n^2-1}\e^4\omega_{n-2}(R_{nanb})^2B(\tfrac{n+3}{2},\tfrac{n-1}{2})\\
&\cdot \int_0^\infty \frac{(y^n)^2[\kappa_2(y^n-T_c)^2+\kappa_1(y^n-T_c)+\kappa_0]}{(1+(y^n-T_c)^2)^{\frac{n-1}{2}}}dy^n+O(\e^{5} \log \frac{\rho}{\e}).
 \end{align*}
 The third term can be estimated by
 \begin{align*}
 \int_{B_\rho^+}(g_{x_0}^{ab}-\delta_{ab})\partial_a\phi \pa_b\phi dy=O(\rho^{8-n}\e^{n-2}).
 \end{align*}
 Thus, we obtain
 \begin{align*}
 &\int_{B^+_\rho}(g_{x_0}^{ab}-\delta_{ab})\partial_a(\U+\phi)\partial_b(\U+\phi) dy\\
 =&\frac{(n-2)^2}{n^2-1}R_{nanb,ab}\Theta_1\e^4+O(\e^5 \log \frac{\rho}{\e})\\
&+\frac{\omega_{n-2}}{n^2-1}\e^4\left\{\tfrac{(n-2)^2(n+1)}{4}(R_{nanb})^2B(\tfrac{n+1}{2},\tfrac{n-1}{2})\int_0^\infty \tfrac{(y^n)^4}{(1+(y^n-T_c)^2)^{\frac{n-1}{2}}}dy^n\right.\\
&\left.-2n(n-2)B(\tfrac{n+3}{2},\tfrac{n-1}{2})(R_{nanb})^2\int_0^\infty\tfrac{(y^n)^2[\kappa_2(y^n-T_c)^2+\kappa_1(y^n-T_c)+\kappa_0]}{(1+(y^n-T_c)^2)^{\frac{n-1}{2}}}dy^n\right\}.
 \end{align*}
  
 We also estimate
 \begin{align*}
&\int_{\Omega_\rho}R_{g_{x_0}}\ubar^2d\mu_{g_{x_0}}\\
=&\int_{B_\rho^+}R_{g_{x_0}}\U^2 dy+2\int_{B_\rho^+}R_{g_{x_0}}\U\phi dy+\int_{B_\rho^+}R_{g_{x_0}}\phi^2dy+O(\e^{n-2}\rho^{N+6-n}).
\end{align*}
Notice that $-\bar \Delta R_{g_{x_0}}=\frac{1}{6}|\overline{W}_{g_{x_0}}|_{g_{x_0}}^2$ at $x_0$ and $R_{g_{x_0}}=O(|y|^2)$ by Proposition \ref{prop:expan_Fermi_metrics}, it yields
\begin{align*}
\int_{B_\rho^+}R_{g_{x_0}}\U^2 dy=&\frac{1}{2}\e^4R_{,nn}\int_{B_{\rho/\e}^+}\frac{|y^n|^2}{(1+|y-T_c\textbf{e}_n|^2)^{n-2}}dy\\
&-\frac{1}{12(n-1)}\e^4|\overline{W}_{g_{x_0}}|_{g_{x_0}}^2\int_{B_{\rho/\e}^+}\frac{|\bar y|^2}{(1+|y-T_c\textbf{e}_n|^2)^{n-2}}dy\\
&+O(\e^5)\int_{B_{\rho/\e}^+}\frac{|y|^3}{(1+|y-T_c\textbf{e}_n|^2)^{n-2}}dy\\
=&\frac{1}{2}\e^4R_{,nn}\Theta_3-\frac{1}{12(n-1)}\e^4|\overline{W}_{g_{x_0}}|_{g_{x_0}}^2\Theta_4+O(\e^5\log \frac{\rho}{\e}),\\
\int_{B_\rho^+}R_{g_{x_0}}\U\phi dy=&\int_{B_\rho^+}O\Big(|y|^2(\e^2+ |y|^2)\Big)\U^2 dy=O(\e^5\rho),\\
\int_{B_\rho^+}R_{g_{x_0}}\phi^2dy=&\int_{B_\rho^+}O\Big(|y|^2(\e^2+|y|^2)^2\Big)\U^2 dy=O(\e^5\rho^3),
\end{align*}
where 
\begin{align*}
\Theta_3=&\int_{\mathbb{R}^n_+}\frac{|y^n|^2}{(1+|y-T_c\textbf{e}_n|^2)^{n-2}}dy\\
=&\frac{\omega_{n-2}}{2}B(\tfrac{n-1}{2},\tfrac{n-3}{2})\int_0^\infty \frac{(y^n)^2}{(1+(y^n-T_c)^2)^{\frac{n-3}{2}}}dy^n,\\
\Theta_4=&\int_{\mathbb{R}^n_+}\frac{|\bar y|^2}{(1+|y-T_c\textbf{e}_n|^2)^{n-2}}dy\\
=&\frac{\omega_{n-2}}{2}B(\tfrac{n+1}{2},\tfrac{n-5}{2})\int_0^\infty \frac{1}{(1+(y^n-T_c)^2)^{\frac{n-5}{2}}}dy^n.
\end{align*}
Also we have
\begin{align*}
\int_{\Omega_\rho\cap \pa M} h_{g_{x_0}}\ubar^2 d\sigma_{g_{x_0}}=&\int_{D_{\rho/\e}}\frac{O(\e^{N}) |y|^{N-1}}{(1+|y-T_c\textbf{e}_n|^2)^{n-2}} d\sigma=O(\e^{n-2}\rho^{N+2-n}).
\end{align*}
Notice  that
$$\Theta_1=\frac{n+1}{4(n-2)}\Theta_3.$$
and $-R_{,nn}-2(R_{nanb})^2=2R_{nanb,ab}$ at $x_0$ by Proposition \ref{prop:expan_Fermi_metrics}, it yields
$$\frac{4(n-2)}{n+1}R_{nanb,ab}\Theta_1+\frac{1}{2}R_{,nn}\Theta_3=-(R_{nanb})^2\Theta_3.$$

Based on the above estimates, we conclude from \eqref{eq:qudratic_root_t} that
$$t_\ast=1+K_\ast \e^4+O(\e^5 \log \frac{\rho}{\e}),$$
where $K_\ast$ is a constant depending on $n,T_c,M,g_{x_0}$. Then we claim that
\begin{align*}
&I_\ast[\ubar]\\
:=&(t_\ast^2-1)E[\ubar]-4(n-1)(n-2)(t_\ast^{\frac{2n}{n-2}}-1)\int_M \ubar^{\frac{2n}{n-2}}d\mu_{g_{x_0}}\\
&+4(n-2)T_c (t_\ast^{\frac{2(n-1)}{n-2}}-1)\int_{\pa M}\ubar^{\frac{2(n-1)}{n-2}}d\sigma_{g_{x_0}}\\
=&\left[\int_{\mathbb{R}_+^n}|\nabla \U|^2 dy -n(n-2)\int_{\mathbb{R}_+^n}\U^{\frac{2n}{n-2}}dy\right.\\
&~~\left.+(n-2)T_c \int_{\mathbb{R}^{n-1}}\U^{\frac{2(n-1)}{n-2}}d\sigma \right]\frac{8(n-1)}{n-2}K_\ast\e^4+O(\e^5 \log \frac{\rho}{\e})\\
=&O(\e^5 \log \frac{\rho}{\e}),
\end{align*}
where the last identity follows from \eqref{energy:bubble2}.

For convenience, noticing that $T_c \geq 0$, we define
\begin{align*}
J(k,l)=&\int_{-\arctan T_c}^{\frac{\pi}{2}}\sin^k\theta\cos^l\theta d\theta,\quad k,l\in\mathbb{N},
\end{align*}
then the iteration formulae of $J(k,l)$ are given by
\begin{align}
J(k+2,l)=&\frac{(-1)^{k+1}T_c^{k+1}}{(l+k+2)(1+T_c^2)^{\frac{k+l+2}{2}}}+\frac{k+1}{l+k+2}J(k,l), \quad k,l\in\mathbb{N},\label{iteration_J1}\\
J(k,l+2)=&\frac{(-1)^{k}T_c^{k+1}}{(l+k+2)(1+T_c^2)^{\frac{k+l+2}{2}}}+\frac{l+1}{l+k+2}J(k,l), \quad k,l\in\mathbb{N}.
\label{iteration_J2}
\end{align}
Under the change of variables $\tan \theta=z-T_c$, we obtain
\begin{align*}
I_0(k)=&\int_0^\infty \frac{(z-T_c)^k}{(1+(z-T_c)^2)^{\frac{n-3}{2}}} dz=J(k,n-5-k), \quad 0\leq k\leq 2,\\
I_1(k)=&\int_0^\infty \frac{(z-T_c)^k}{(1+(z-T_c)^2)^{\frac{n-1}{2}}} dz=J(k,n-3-k), \quad 0\leq k\leq 4,\\
I_2(k)=&\int_0^\infty \frac{(z-T_c)^k}{(1+(z-T_c)^2)^{\frac{n+1}{2}}} dz=J(k,n-1-k), \quad 0\leq k\leq 4.
\end{align*}
Therefore, putting these facts together, together with \eqref{eq:S_c_n=4} we conclude that
\begin{align}
&I[t_\ast \ubar]\no\\
=&I[\ubar]+I_\ast[\ubar]\no\\
=&\frac{4(n-1)}{(n-2)}\left\{\int_{\mathbb{R}^n_+}|\nabla \U|^2 dy+O(\e^5\log(\tfrac{\rho}{\e})) \right.\no\\
&+\left[n^2B(\tfrac{n+3}{2},\tfrac{n-1}{2})\int_0^\infty \frac{(\kappa_2(y^n-T_c)^2+\kappa_1(y^n-T_c)+\kappa_0)^2}{(1+(y^n-T_c)^2)^{\frac{n-1}{2}}}dy^n\right.\no\\
&\quad+B(\tfrac{n+3}{2},\tfrac{n-3}{2})\int_0^\infty \frac{(2\kappa_2(y^n-T_c)+\kappa_1)^2}{(1+(y^n-T_c)^2)^{\frac{n-3}{2}}} dy^n\no\\
&\quad-2nB(\tfrac{n+3}{2},\tfrac{n-1}{2})\int_0^\infty\tfrac{(2\kappa_2(y^n-T_c)+\kappa_1)(\kappa_2(y^n-T_c)^2+\kappa_1(y^n-T_c)+\kappa_0)(y^n-T_c)}{(1+(y^n-T_c)^2)^{\frac{n-1}{2}}} dy^n \no\\
&\left.\quad\left.-n^2B(\tfrac{n+3}{2},\tfrac{n+1}{2})\int_0^\infty \tfrac{(\kappa_2(y^n-T_c)^2+\kappa_1(y^n-T_c)+\kappa_0)^2}{(1+(y^n-T_c)^2)^{\frac{n+1}{2}}}dy^n \right]\frac{\omega_{n-2}}{n^2-1}(R_{nanb})^2\e^4\right.\no\\
&+\frac{\omega_{n-2}}{n^2-1}\e^4\left[\frac{(n-2)^2(n+1)}{4}(R_{nanb})^2B(\tfrac{n+1}{2},\tfrac{n-1}{2})\int_0^\infty \tfrac{(y^n)^4}{(1+(y^n-T_c)^2)^{\frac{n-1}{2}}}dy^n\right.\no\\
&\left.\left.-2n(n-2)B(\tfrac{n+3}{2},\tfrac{n-1}{2})(R_{nanb})^2\int_0^\infty\tfrac{(y^n)^2(\kappa_2(y^n-T_c)^2+\kappa_1(y^n-T_c)+\kappa_0)}{(1+(y^n-T_c)^2)^{\frac{n-1}{2}}}dy^n\right] \right\}\no\\
&-\e^4(R_{nanb})^2\frac{\omega_{n-2}}{2}B(\tfrac{n-1}{2},\tfrac{n-3}{2})\int_0^\infty \frac{(y^n)^2}{(1+(y^n-T_c)^2)^{\frac{n-3}{2}}}dy^n \no\\
&-\frac{1}{12(n-1)}\e^4|\overline{W}_{g_{x_0}}|_{g_{x_0}}^2\Theta_4+O(\e^5\log(\tfrac{\rho}{\e})) -4(n-1)(n-2)\left[\int_{B_\rho^+}\U^{\frac{2n}{n-2}}dy\right.\no\\
&+\left.\tfrac{n(n+2)\omega_{n-2}}{(n-2)^2(n^2-1)}B(\tfrac{n+3}{2},\tfrac{n+1}{2})(R_{nanb})^2\e^4\int_0^\infty \tfrac{(\kappa_2(y^n-T_c)^2+\kappa_1(y^n-T_c)+\kappa_0)^2}{(1+(y^n-T_c)^2)^{\frac{n+1}{2}}}dy^n\right]\no\\
&+4(n-2)T_c\left[\int_{D_\rho}\U^{\frac{2(n-1)}{n-2}}d\sigma+\tfrac{n\omega_{n-2}B\lp\tfrac{n+3}{2},\tfrac{n-1}{2}\rp(\kappa_2 T_c^2-\kappa_1 T_c+\kappa_0)^2\e^4(R_{nanb})^2}{(n-2)^2(n+1)(1+T_c^2)^{\frac{n-1}{2}}}\right]\no \\
=& S_c+\omega_{n-2}B(\tfrac{n+3}{2},\tfrac{n-1}{2})\kappa\mathcal{Q}\kappa^{\top}\e^4(R_{nanb})^2-\tfrac{\Theta_4}{12(n-1)}|\overline{W}_{g_{x_0}}|_{g_{x_0}}^2\e^4+O(\e^5\log(\tfrac{\rho}{\e})) , \label{est:error_final_higher_dim}
\end{align}
where
\begin{align*}
\kappa=&(\kappa_2,\kappa_1,\kappa_0,1), \quad\mathcal{Q}=(\mathcal{Q}_{ij})_{1\leq i,j \leq 4}, \quad \mathcal{Q}_{ij}=\mathcal{Q}_{ji},\\
\mathcal{Q}_{11}=&\tfrac{4n}{(n-2)(n+1)}\left[(n-4)I_1(4)-(n-1)I_2(4)+\frac{8}{n-3}I_0(2)+\frac{T_c^5}{(1+T_c^2)^{\frac{n-1}{2}}}\right]\\
=&\frac{8n}{(n-2)(n+1)}\left[J(4,n-7)+\frac{4}{n-3}J(2,n-7)\right], \\
\mathcal{Q}_{12}=&\tfrac{4n}{(n-2)(n+1)}\left[(n-3)I_1(3)-(n-1)I_2(3)+\frac{4}{n-3}I_0(1)-\frac{T_c^4}{(1+T_c^2)^\frac{n-1}{2}}\right]\\
=&\frac{8n}{(n-2)(n+1)}\left[J(3,n-6)+\frac{2}{n-3}J(1,n-6)\right], \\
\mathcal{Q}_{13}=&\frac{4n}{(n-2)(n+1)}\left[(n-2)I_1(2)-(n-1)I_2(2)+\frac{T_c^3}{(1+T_c^2)^\frac{n-1}{2}}\right]\\
=&\frac{8n}{(n-2)(n+1)}J(2,n-5), \\
\mathcal{Q}_{14}=&-\frac{4n}{n+1}\left[I_1(4)+2T_cI_1(3)+T_c^2I_1(2)\right]\\
=&-\frac{4n}{n+1}\left[J(4,n-7)+2T_cJ(3,n-6)+T_c^2J(2,n-5)\right],\\
\mathcal{Q}_{22}=&\tfrac{4n}{(n-2)(n+1)}\left[(n-2)I_1(2)-(n-1)I_2(2)+\frac{2}{n-3}I_0(0)+\frac{T_c^3}{(1+T_c^2)^{\frac{n-1}{2}}}\right]\\
=&\frac{8n}{(n-2)(n+1)}\left[J(2,n-5)+\frac{1}{n-3}J(0,n-5)\right], \\
\mathcal{Q}_{23}=&\frac{4n}{(n-2)(n+1)}\left[(n-1)I_1(1)-(n-1)I_2(1)-\frac{T_c^2}{(1+T_c^2)^{\frac{n-1}{2}}}\right]\\
=&\frac{8n}{(n-2)(n+1)}J(1,n-4), \\
\mathcal{Q}_{24}=&-\frac{4n}{n+1}\left[I_1(3)+2T_cI_1(2)+T_c^2I_1(1)\right]\\
=&-\frac{4n}{n+1}\left[J(3,n-6)+2T_cJ(2,n-5)+T_c^2J(1,n-4)\right],\\
\mathcal{Q}_{33}=&\frac{4n}{(n-2)(n+1)}\left[nI_1(0)-(n-1)I_2(0)+\frac{T_c}{(1+T_c^2)^{\frac{n-1}{2}}}\right]\\
=&\frac{8n}{(n-2)(n+1)}J(0,n-3), \\
\mathcal{Q}_{34}=&-\frac{4n}{n+1}\left[I_1(2)+2T_cI_1(1)+T_c^2I_1(0)\right]\\
=&-\frac{4n}{n+1}\left[J(2,n-5)+2T_cJ(1,n-4)+T_c^2J(0,n-3)\right],\\
\mathcal{Q}_{44}=&\Big[-\frac{2}{n-3}(I_0(2)+2T_cI_0(1)+T_c^2I_0(0))\\
&~~+ (I_1(4)+4T_cI_1(3)+6T_c^2I_1(2)+4T_c^3I_1(1)+T_c^4I_1(0))\Big]\frac{2n(n-2)}{n+1} \\
=&\Big[-\frac{2}{n-3}(J(2,n-7)+2T_cJ(1,n-6)+T_c^2J(0,n-5)) \\
&~~+(J(4,n-7)+4T_cJ(3,n-6)+6T_c^2J(2,n-5))+4T_c^3J(1,n-4) \\
&~~+T_c^4J(0,n-3))\Big]\frac{2n(n-2)}{n+1}
\end{align*}
by virtue of \eqref{iteration_J2}.

By \eqref{est:error_final_higher_dim} and \eqref{ieq:goal_less}, it remains to prove the following
\begin{claim}\label{seek_vector}
There exists a vector $\kappa=\kappa(n,c) \in \mathbb{R}^4$ with its fourth element equal to $1$ such that $\kappa \mathcal{Q}\kappa^\top<0$.
\end{claim}
To see this, we choose two non-singular matrices 
\begin{equation}\label{first_nonsing_matrix}
\mathcal{S}_1=\mathrm{diag}\left\{1,1,1,-\frac{2}{n-2}\right\}
\end{equation}
and
\begin{align}\label{second_nonsing_matrix}
\mathcal{S}_2=\left(\begin{array}{lllc}
1&0 &0 &-1  \\
 0 &1&0&-2T_c \\
  0 & 0  &1 &-T_c^2 \\
      0 & 0     & 0&1
\end{array}
\right).
\end{align}
Then we obtain a symmetric matrix $\overline{\mathcal{Q}}=\frac{(n-2)(n+1)}{8n}\mathcal{S}_2^\top \mathcal{S}_1^\top\mathcal{Q}\mathcal{S}_1\mathcal{S}_2$, where 
\begin{align*}
\overline{\mathcal{Q}}_{11}&=J(4,n-7)+\tfrac{4}{n-3}J(2,n-7),\\
\overline{\mathcal{Q}}_{12}&=J(3,n-6)+\tfrac{2}{n-3}J(1,n-6),\\
\overline{\mathcal{Q}}_{13}&=J(2,n-5),\\
\overline{\mathcal{Q}}_{14}&=-\tfrac{4}{n-3}J(2,n-7)-\tfrac{4T_c}{n-3}J(1,n-6),\\
\overline{\mathcal{Q}}_{22}&=J(2,n-5)+\tfrac{1}{n-3}J(0,n-5),\\
\overline{\mathcal{Q}}_{23}&=J(1,n-4),\\
\overline{\mathcal{Q}}_{24}&=-\tfrac{2}{n-3}J(1,n-6)-\tfrac{2T_c}{n-3}J(0,n-5),\\
\overline{\mathcal{Q}}_{33}&=J(0,n-3),\\
\overline{\mathcal{Q}}_{34}&=0,\\
\overline{\mathcal{Q}}_{44}&=\tfrac{2}{n-3}\left[J(2,n-7)+2T_cJ(1,n-6)+T_c^2J(0,n-5)\right].
\end{align*}
For any $a \in \mathbb{R}$ satisfying $7a^2-8a+2<0$, by choosing $\mathcal{V}=(a,T_c,0,1)$ we find 
\begin{align*}
&\mathcal{V}\overline{\mathcal{Q}}\mathcal{V}^\top\\
=&a^2\left[J(4,n-7)+\frac{4}{n-3}J(2,n-7)\right]+2aT_c\left[J(3,n-6)+\frac{2}{n-3}J(1,n-6)\right]\\
&+2a\left[-\frac{4}{n-3}J(2,n-7)-\frac{4T_c}{n-3}J(1,n-6)\right]\\
&+T_c^2\left[J(2,n-5)+\frac{1}{n-3}J(0,n-5)\right]\\
&+2T_c\left[-\frac{2}{n-3}J(1,n-6)-\frac{2T_c}{n-3}J(0,n-5)\right]\\
&+\frac{2}{n-3}\left[J(2,n-7)+2T_cJ(1,n-6)+T_c^2J(0,n-5)\right]\\
=&-\frac{(a-1)^2}{n-3}\frac{T_c^3}{(1+T_c^2)^{\frac{n-3}{2}}}+\frac{7a^2-8a+2}{n-3}J(2,n-7)<0,
\end{align*}
where we have used \eqref{iteration_J1} in the last identity.

Therefore, we can choose $\kappa=-\frac{n-2}{2}\mathcal{V}\mathcal{S}_2^\top \mathcal{S}_1^\top$ whose fourth element equals $1$, such that $\kappa \mathcal{Q}\kappa^\top<0$. This proves the Claim.
\end{proof}

\appendix

\section{Appendix}\label{Appen}
\renewcommand{\theequation}{A-\arabic{equation}}
\setcounter{equation}{0}
\renewcommand{\thetheorem}{A-\arabic{theorem}}
\setcounter{theorem}{0}

Our first purpose is to prove that $I$ satisfies $(PS)$ condition for energy level below $S_c$.  For clarity, we restate Lemma \ref{lem:compactness} here.

\begin{lemma}[\textbf{Compactness}]\label{lem:appen_compactness} Suppose that $Y(M,\partial M,[g_0])>0$. Let $\{u_i;i \in \mathbb{N}\}$ be a sequence of functions in $H^1(M,g_0)$ satisfying $I[u_i]\to L<S_c$ and 
$$\max_{v\in H^1(M,g_0)\backslash\{0\}}\frac{|I'[u_i](v)|}{\|v\|_{H^1(M,g_0)}}\to 0 \quad\mathrm{~~as~~} i \to \infty.$$
Then after passing to a subsequence, either (i) $\{u_i\}$ strongly converges in $H^1(M,g_0)$ to some positive solution $u$ of \eqref{eq:main_eq_on_M} or (ii) $\{u_i\}$ strongly converges to $0$ in $H^1(M,g_0)$. 
\end{lemma}
\begin{proof}
Since $Y(M,\pa M,[g_0])>0$, we define $\langle u,v\rangle:=\frac{4(n-1)}{n-2}\int_{M}\langle \nabla u,\nabla v\rangle_{g_0}d\mu_{g_0}+\int_{M} R_{g_0}uv d\mu_{g_0} +2(n-1)\int_{\pa M} h_{g_0}uv d\sigma_{g_0}$ as an inner product of $H^1(M,g_0)$. The norm defined by $\|u\|=\langle u,u \rangle^{1/2}$ is equivalent to $\|u\|_{H^1(M,g_0)}$. Based on Han-Li \cite[Lemma 1.2]{han-li2}, we only need to show that if $\{u_i\}\subset H^1(M,g_0)$ weakly converges to some nontrivial solution $u$ of \eqref{eq:main_eq_on_M} , then $u_i \to u$ in $H^1(M,g_0)$ as $i \to \infty$.

By the Sobolev embedding and trace inequalities, up to a subsequence we have 
\begin{align*}
u_i \rightharpoonup u \quad&\mathrm{~~in~~} H^1(M,g_0);\\
u_i \to u \quad &\mathrm{~~in~~} L^2(M,g_0) \mathrm{~~and~~} L^2(\pa M,g_0).
\end{align*}
For any $v\in H^1(M,g_0)$ we have
\begin{align}\label{PS_1}
&\frac{1}{2}I'[u_i-u]v \no\\
=&\langle u_i-u,v\rangle-4n(n-1)\int_{M}(u_i-u)_+^\frac{n+2}{n-2}vd\mu_{g_0}-\tfrac{4(n-1)}{n-2}c\int_{\pa M}(u_i-u)_+^\frac{n}{n-2}vd\sigma_{g_0}\no\\
\to& 0, \quad \mathrm{as~~} i\to \infty, 
\end{align}
which implies that $I'[u_i-u]\to 0$ as $i\to \infty$.

By the Lebesgue dominant convergence theorem, we have
\begin{align*}
I[u_i-u] =&I[u_i]+\int_0^1\frac{d}{dt}I[u_i-tu]dt \\
=&I[u_i]-\int_0^1 I'[u_i-tu]u dt \\
\to& L-\int_0^1 I'[(1-t)u]u dt, \mathrm{~~as~~} i\to\infty. 
\end{align*}

Since $u$ is a nontrivial smooth solution of \eqref{eq:main_eq_on_M}, we obtain
$$
\left.\frac{d}{dt}\right|_{t=1}I[tu]=0,
$$
and then
\begin{align*}
&\frac{d}{dt}I[tu]\\
=&2t\left[E[u]-4n(n-1)t^{\frac{4}{n-2}}\int_{M}u^{\frac{2n}{n-2}}d\mu_{g_0}-\frac{4(n-1)}{n-2}ct^\frac{2}{n-2}\int_{\pa M}u^{\frac{2(n-1)}{n-2}}d\sigma_{g_0}\right]\\
=&2t(1-t^{\frac{2}{n-2}})\left[E[u]+t^{\frac{2}{n-2}}4n(n-1)\int_M u^{\frac{2n}{n-2}}d\mu_{g_0}\right].
\end{align*}
Thus,
$$\frac{d}{dt}I[tu]=I'[tu]u\geq0,\quad  \forall~ t \in [0,1].$$
From these we conclude that
\begin{equation}\label{PS_2}
\lim_{i\to \infty}I[u_i-u]\leq L<S_c.
\end{equation}
Therefore, thanks to \eqref{PS_1} and \eqref{PS_2}, we can apply the same argument of Case \textit{(ii)} in \cite[Lemma 1.2]{han-li2} to $u_i-u$ and then obtain  $u_i \to u$ in $H^1(M,g_0)$ as $i \to \infty$.
\end{proof}

Next we obtain the following refined estimate of \cite[Proposition 5.5]{Chen-Sun} (see also \cite[Corollary 2.6]{Brendle-Chen}), which was used in Section \ref{Sect5}.
\begin{proposition}\label{app:prop:1}
Let $(M,g_0)$ be a compact Riemannian manifold of dimension $n\geq 3$ with umbilic boundary $\pa M$ and $c \in \mathbb{R}$. Then there exists a positive dimensional constant $\lambda^\ast$ (independent of $T_c$), such that 
\begin{align*}
&\lambda^\ast\e^{n-2}\sum_{a,b=1}^{n-1}\sum_{|\a|=2}^{d}|\pa^\alpha h_{ab}|^2\int_{B^+_{\rho}(0)}(\e+|y|)^{2|\a|+2-2n}dy\\
\leq& \int_{B^+_{\rho}(0)}\sum_{i,k,l=1}^n Q_{ik,l}Q_{ik,l}dy
\end{align*}
for all $\rho\geq 2(1+|T_c|)\e$.
\end{proposition}
\begin{proof}
For $r>0$, let $U_r \subset \mathbb{R}_+^n$ be an open ball of radius $r/4$ centered at $(3r/2) \textbf{e}_n$. Denote by $\eta(y/r)$ a smooth cut-off function such that $\eta=1$ in $U_r$, $\eta=0$ in $B_{2r}^+(0)\setminus \overline{B_r^+}$. Since $\pa M$ is umbilic, we can apply \cite[Proposition 2.4]{Brendle-Chen} to show
\begin{align*}
C(n) \sum_{a,b=1}^{n-1}\sum_{|\a|=2}^{d}|\pa^\alpha h_{ab}|^2 r^{2|\alpha|-4+n}\leq \int_{U_r}\sum_{i,j,k,l=1}^n |Z_{ijkl}(y)|^2 dy
\end{align*}
for all $r>0$.

Observe that 
\begin{align*}
&\frac{1}{4}\sum_{i,j,k,l=1}^n|Z_{ijkl}|^2\\
=&\sum_{i,j,k,l=1}^n\pa_j(\U^{-1}Q_{ik,l})Z_{ijkl}+\frac{2}{n-2}\sum_{i,j,k,l=1}^n\U^{-2}\pa_k\U Q_{il,j}Z_{ijkl}.
\end{align*}
Multiplying the above equation by $\eta(y/r)$ and integrating over $\R^n_+$, we obtain
\begin{align*}
\frac{1}{4}\int_{\R^n_+}\sum_{i,j,k,l=1}^n|Z_{ijkl}|^2\eta(\frac{y}{r})dy=&-\int_{\R^n_+}\sum_{i,j,k,l=1}^n\U^{-1}Q_{ik,l}\pa_j\left[Z_{ijkl}\eta(\frac{y}{r})\right]dy\\
&+\frac{2}{n-2}\int_{\R^n_+}\sum_{i,j,k,l}^n\U^{-2}\pa_k\U Q_{il,j}Z_{ijkl}\eta(\frac{y}{r})dy.
\end{align*}
Applying H\"{o}lder's inequality, we have
\begin{align*}
&\int_{U_r}\sum_{i,j,k,l=1}^n|Z_{ijkl}|^2dy\\
\leq& C\e^{-\frac{n-2}{2}}(r+(1+|T_c|\e))^{n-2}\left(\sum_{|\alpha|=2}^d\sum_{i,k=1}^n|\pa^\alpha h_{ik}|^2r^{2|\alpha|-6+n}\right)^{\frac{1}{2}}\\
&\cdot \left(\int_{B^+_{2r}(0)\backslash B^+_r(0)}\sum_{i,k,l=1}^n|Q_{ik,l}(y)|^2dy\right)^{\frac{1}{2}}\\
\leq&Cr^{n-3}\left(\sum_{|\alpha|=2}^d\sum_{i,k=1}^n|\pa^\alpha h_{ik}|^2r^{2|\alpha|-4+n}\right)^{\frac{1}{2}}\left(\int_{B^+_{2r}(0)\backslash B^+_r(0)}\sum_{i,k,l=1}^n|Q_{ik,l}(y)|^2dy\right)^{\frac{1}{2}}
\end{align*}
for all $r\geq \e(1+ |T_c|)$. For such $\e$, one can combine with \cite[Proposition 2.4]{Brendle-Chen} to get the conclusion.
\end{proof}


\end{document}